\documentclass[leqno,twoside,11pt]{amsart}
\usepackage{latexsym, esint, url, amssymb, yfonts, amsmath}
\usepackage{graphicx,color}

\usepackage{caption}
\usepackage{subcaption}
\usepackage{array}

\theoremstyle{plain}

\def\endproof{\hspace*{\fill}\mbox{\ \rule{.1in}{.1in}}\medskip }

\newtheorem{theorem}{Theorem}[section]
\newtheorem{lemma}[theorem]{Lemma}
\newtheorem{proposition}[theorem]{Proposition}

\theoremstyle{definition}

\newtheorem{example}[theorem]{Example}

\numberwithin{equation}{section}
\numberwithin{figure}{section}

\def\Det{{\mathcal{D}et}\,}
\def\cc{{\rm curl}\,{\rm curl}\,} 
\def \Sym{{\rm sym}\,} 
\def\R{{\mathbb R}}

\definecolor{Green}{rgb}{0, 0.65,0}

\begin{document}
\title[Visualization of the anomalous Monge-Amp\`ere solutions]
{Visualization of the convex integration solutions \\ to the Monge-Amp\`ere equation}
\author{Luca Codenotti}
\author{Marta Lewicka}\footnote{University of Pittsburgh, Department of Mathematics,
301 Thackeray Hall, Pittsburgh, PA 15260, USA. \\ E-mail: luc23@pitt.edu, lewicka@pitt.edu }

\begin{abstract}
In this article, we implement the algorithm based on the convex
integration result proved in \cite{LP}
and obtain visualizations of the first iterations of the Nash-Kuiper
scheme, approximating the anomalous solutions to the Monge-Amp\`ere
equation in two dimensions.
\end{abstract}

\maketitle

\section{Introduction}

The purpose of this paper is to implement the technique of convex
integration to obtain and analyze the visualizations of the H\"older regular
$\mathcal{C}^{1,\alpha}$ solutions to the Monge-Amp\`ere equation:
\begin{equation} \label{MA}
\Det \nabla^2 v  \doteq  -\frac 12 \cc  (\nabla v \otimes \nabla v) =
f \qquad \mbox{ in } \Omega\subset\R^2.
\end{equation}
We rely on the results of the paper \cite{LP}, where the Nash-Kuiper
iteration was shown to be adaptable to prove flexibility of the very weak solutions
to (\ref{MA}) in the regularity regime $0<\alpha<\frac{1}{7}$. More
precisely, the following holds:

\begin{theorem}\label{weakMA}   \cite{LP}
Let $f \in L^{7/6} (\Omega)$ on an open, bounded, simply connected domain
$\Omega\subset\mathbb{R}^2$.  Fix an exponent $\alpha\in (0,\frac{1}{7}).$
Then the set of functions
$v\in\mathcal{\mathcal{C}}^{1,\alpha}(\bar\Omega)$ that satisfy
(\ref{MA}) in the sense of distributions,
is dense in the space $\mathcal{\mathcal{C}}^0(\bar\Omega)$. 
Namely, for every $v_0\in \mathcal{\mathcal{C}}^0(\bar\Omega)$ there exists a sequence
$v_n\in\mathcal{\mathcal{C}}^{1,\alpha}(\bar\Omega)$ converging
uniformly to $v_0$ and solving  (\ref{MA}).
When $f\in L^p(\Omega)$ and $p\in (1,\frac{7}{6})$, the same result is true
for any exponent $\alpha\in (0,1- \frac{1}{p})$.
\end{theorem} 

We note that by a straightforward approximation argument, if $v\in W^{2,2}$ then $\Det
\nabla^2 v = \det \nabla^2 v$ a.e. in $\Omega$, 
where $\nabla^2 v\in L^2(\Omega)$ stands for the standard Hessian matrix field of $v$.  
It has been shown \cite{LMP, Pak} that in this case, condition $f\geq c>0$ in $\Omega$
(respectively, $f\equiv 0$ in $\Omega$) implies that $v$ must be
$\mathcal{C}^1$ and globally convex or concave (respectively, developable).
Likewise \cite{LP, LP2}, the flexibility result of Theorem
\ref{weakMA} has its rigidity counterpart: if
$v\in\mathcal{C}^{1,\alpha}(\bar\Omega)$ with $\alpha>\frac{2}{3}$
solves (\ref{MA}) for the right hand side that is positive Dini
continuous (respectively, $f\equiv 0$ in $\Omega$) then $v$ is a convex
Alexandrov solution to $\det\nabla^2v=f$ (respectively, $v$ is developable).

\smallskip

For a better understanding of the above statements and their relation
to nonlinear elasticity of thin films, we refer to \cite{LP}. 
We now recall the connection between solutions to \eqref{MA} and the  problem of
isometric immersion of Riemannian metrics. Consider the
one-parameter family of metrics $h_\epsilon=\mbox{Id}_2+2\epsilon^2A$
on $\Omega$, where $A:\Omega\to\R^{2\times 2}_{sym}$ is a perturbation
tensor that satisfies: 
\begin{equation}\label{curlsol}
-\cc A = f,
\end{equation} 
for example one can take: $A=-\Delta^{-1}(f)\mbox{Id}_2$. 
Consider also the pull-back metrics
$(\nabla\phi_\epsilon)^T\nabla\phi_\epsilon$ of the family of
deformations $\phi_\epsilon = id +\epsilon ve_3 + \epsilon^2w:\Omega\to\R^3$.
Computing the Gauss curvature of $h_\epsilon$:
$$\kappa(h_\epsilon) = -\epsilon^2\cc A + o(\epsilon^2)$$
and the curvature of the pull-back metrics: $\kappa((\nabla\phi_\varepsilon)^T\nabla\phi_\varepsilon) =
-\epsilon^2\cc\big(\frac{1}{2}\nabla v\otimes\nabla v + \Sym \nabla w\big) +
o(\epsilon^2)$, we see that (\ref{MA}) can be interpreted as
requesting the out-of-plane displacement $v$ to be amenable to matching by a
higher order in-plane displacement $w$, to 
achieve the prescribed Gaussian curvature $f$, at its highest order
term. Indeed, on a simply connected domain the
kernel of the differential operator ``$\mbox{curl}\,\mbox{curl}$'' consists  
of the fields of the form $\Sym \nabla w$. Thus, equation (\ref{MA}) is equivalent to:
\begin{equation}\label{MA2} 
\frac 12 \nabla v \otimes \nabla v + \Sym \nabla w = A \qquad
\mbox{ in } \Omega,
\end{equation}
where $A$ satisfies (\ref{curlsol}) and $w:\Omega\to\R^2$ is an auxiliary variable (a Lagrange
multiplier). In the same spirit, equation (\ref{MA2}) can be interpreted as the
``infinitesimal isometry'' constraint, namely: the family 
$h_\epsilon$ should coincide, up to terms of order $\epsilon^2$, with the pull-back metrics
$(\nabla\phi_\epsilon)^T\nabla\phi_\epsilon$ defined above.

\smallskip

In order to construct $\mathcal{C}^{1,\alpha}$ solutions of (\ref{MA2}) with $v$ approximating a given
continuous $v_0$ (by density, it suffices to work with $v_0$ smooth),
it has been shown in \cite{LP} that a Nash-Kuiper
convex integration method can be employed. This method \cite{gromov,
  kuiper} modifies an initial ``short infinitesimal isometry'' $(v_0, w_0)$, i.e. a couple for whom
(\ref{MA2}) is satisfied with an inequality (hence a subsolution) rather than the equality,
towards an exact solution, in successive small steps.  Technical similarities with the isometric immersion
construction \cite{Nash1, Nash2, kuiper, CDS}
are based on the presence of the quadratic terms $(\nabla \phi_\epsilon)^T \nabla \phi_\epsilon$ and
$\nabla v \otimes \nabla v$ in both PDEs. For parallel application of
the convex integration technique to constructing the 
energy-dissipative solutions of the Euler equations and to the recent
resolution of the Onsager's conjecture, we refer to \cite{DS0, DS, B1, Is1}.

\smallskip

As we have mentioned, Theorem \ref{weakMA} follows from the main result in \cite{LP} below: 
 
\begin{theorem}\label{w2oi-hld}\cite{LP}
Let $\Omega\subset\mathbb{R}^2$ be an open and bounded domain. Let
$v_0\in\mathcal{C}^1(\bar\Omega)$, $w_0\in\mathcal{C}^1(\bar\Omega,\mathbb{R}^2)$ and $A \in
\mathcal{C}^{0,\beta} (\bar \Omega,  \R^{2\times 2}_{sym})$, for some $\beta\in (0,1)$, be such that: 
\begin{equation}\label{wA-inequ-holder}
\exists c_0>0 \qquad A - \big(\frac 12 \nabla v_0
\otimes \nabla v_0 + \Sym \nabla w_0\big)  \geq  c_0 {\rm Id}_2 \qquad \mbox{ in } \bar\Omega.
\end{equation} 
Then, for every exponent $\alpha\in \big(0,\min\big\{\frac{1}{7},
\frac{\beta}{2}\big\}\big),$
there exist sequences: $v_n \in \mathcal{C}^{1,\alpha}(\bar \Omega)$
which converges uniformly to $v_0$,
and $w_n \in \mathcal{C}^{1,\alpha} (\bar \Omega,\R^2)$ which
converges uniformly to $w_0$, satisfying:
\begin{equation}\label{2oi-holder} 
 \frac 12 \nabla v_n \otimes \nabla v_n + \Sym \nabla w_n = A \qquad \mbox{ in } \bar\Omega.
\end{equation}
\end{theorem}


In this paper, we implement the proof of Theorem \ref{w2oi-hld} (for
$f$ and $v_0$ as in Theorem \ref{weakMA} and specially prepared $A$ and
$w_0$) into an explicit algorithm, leading to the visualizations of the anomalous
solutions to the Monge-Amp\`ere equation (\ref{MA}). These are
solutions failing to possess the indicated geometrical 
structural properties valid in the rigid regime $v\in W^{2,2}$ or
$v\in \mathcal{C}^{1,2/3+}$:
\begin{itemize}
\item nonconvex solutions, approximating a saddle $v_0(x,y) = x^2-y^2$,
  when $f\equiv 1$,
\item nondevelopable solutions, approximating a convex paraboloid $v_0
  (x,y) =  x^2+y^2$, when $f\equiv 0$,
\item solutions uniformly close to $v_0(x,y) = x^2+y^2$,  when $f\equiv -1$.
\end{itemize}
We remark that numerical implementation of convex
integration construction for the isometric immersion problem has been
reported in \cite{BJLT, BJLT2}, leading among others to the
remarkable images of flat torus embedded in $\R^3$.

\smallskip

The paper is organized as follows. In sections \ref{sec1}, \ref{sec4} and \ref{sec5}
we revisit the proof of Theorem \ref{w2oi-hld} and explicitly trace
all the constants appearing in the arguments in \cite{LP} for an optimal performance of the numerical
algorithm. Without loss of generality, we will assume that
$0\in\Omega$. In sections \ref{Numerical1} and \ref{num2_sec} we
implement the indicated technique and visualise the first few
approximations of anomalous solutions to (\ref{MA}) in case of $f$
being positive, negative and equivalently equal $0$, respectively.
The implementation is done in MATLAB, where for the detailed H\"older
continuous approximations in section \ref{num2_sec} we use of the Python package
mpmath \cite{mpmath}, allowing the user to define floating point
arithmetics up to an arbitrary precision. 

\bigskip

\noindent{\bf Acknowledgments.}
The authors have been partially supported by the NSF awards DMS-1406730 and DMS-1613153. 

\section{The $\mathcal{C}^1$ convergence and construction of corrugations}\label{sec1}

We start with a weaker version of Theorem \ref{w2oi-hld}. Define three unit vectors:
\begin{equation*}
 \eta_1= (1,0),\quad \eta_2= \frac{1}{\sqrt{5}}(1,2),\quad \eta_3= \frac{1}{\sqrt{5}}(1,-2).
\end{equation*}

\begin{theorem}\label{Approx2}
Let $\Omega\subset\R^2$ be an open bounded domain. Given $v_0\in \mathcal{C}^\infty
 ( \bar{\Omega}) $, $w_0\in \mathcal{C}^\infty (\bar{\Omega},\R^2)$
 and $A\in \mathcal{C}^\infty (\bar{\Omega},\R^{2\times2}_{sym})$, we have the following.
For any $\epsilon>0$ there exist $v\in \mathcal{C}^1( \bar{\Omega}) $
and $w\in \mathcal{C}^1 (\bar{\Omega},\R^2)$ that satisfy: 
\begin{equation}\label{result}
\|v-v_0\|_0 \leq \epsilon \quad \mbox{ and } \quad 
\frac{1}{2}\nabla v\otimes\nabla v + \mathrm{sym} \nabla w = A.
\end{equation}
Moreover, if with some smooth positive functions
$\phi_k\in\mathcal{C}^\infty(\bar\Omega)$ there holds: 
\begin{equation}\label{DefectBasis}
   A-\big( \frac{1}{2}\nabla v_0\otimes\nabla v_0 +
    \mathrm{sym} \nabla w_0\big) = \sum_{k=1}^3 \phi_k
  \eta_k\otimes\eta_k,\qquad \phi_k\geq d_0>0 \quad \mbox{ in }
  \bar{\Omega}.
\end{equation}
then it is possible to have the field $w$ in (\ref{result}) 
additionally satisfy: $\|w-w_0\|_{0}\leq \epsilon$.
The vector field $(v,w)$ is obtained as a
$\mathcal{C}^1(\bar\Omega)$-limit of smooth fields
$v_k\in\mathcal{C}^\infty(\bar\Omega)$ and
$w_k\in\mathcal{C}(\bar\Omega,\R^2)$ such that $\|v_k-v_0\|_0\leq
\epsilon$ and $\|w_k-w_0\|_0\leq \epsilon$ for all $k\geq 1$.
\end{theorem}

As noted below, condition (\ref{DefectBasis}) implies
positive definiteness of the left hand side:
$D_0\doteq A-\big( \frac{1}{2}\nabla v_0\otimes\nabla v_0 +
    \mathrm{sym} \nabla w_0\big)
\geq d_0 \mathrm{Id}_2$ in $\bar\Omega$. We now recall the convex
integration construction which allows to decrease such positive definite defect, 
in each of the $\eta_k\otimes \eta_k$  directions, through oscillatory
perturbations in $v$ and $w$.
Let $\eta$ be a unit vector and $\phi$ a smooth positive function on $\bar\Omega$.
Given two fields $v$ and $w$, we want to adjust them  to
$\tilde{v}$ and $\tilde{w}$ so that:
\begin{equation}\label{ErrAlongEta}
\Big(A- \big(\frac{1}{2}\nabla \tilde{v}\otimes\nabla \tilde{v} + \mathrm{sym}
 \nabla \tilde{w}\big)\Big) - \Big(A-\big(\frac{1}{2}\nabla v\otimes\nabla v + \mathrm{sym}
 \nabla w + \phi \eta\otimes\eta\big)\Big) \sim 0.
\end{equation}
Firstly, define $\tilde{v} \doteq v+\frac{1}{\lambda}f(x,\lambda
x\cdot\eta)$ by adding oscillations of amplitude $\frac{1}{\lambda}$ and
frequency $\lambda$, with $\lambda\to\infty$ in the ultimate limiting process. Secondly,
decompose perturbation of $w$ in the two directions $\nabla v$ and $\eta$, and write 
$\tilde{w} \doteq w + \frac{1}{\lambda}g(x,\lambda x\cdot\eta)\nabla v +
\frac{1}{\lambda}h(x,\lambda x\cdot\eta)\eta$.  
Each of $f(x,t)$, $g(x,t)$, and $h(x,t)$ should be
bounded and $1$-periodic in the variable $t$. The expansion:
\begin{equation*}
\begin{split}
\frac{1}{2}\nabla \tilde{v}\otimes\nabla \tilde{v} + \mathrm{sym} \nabla \tilde{w} = 
 & \frac{1}{2}\nabla v\otimes\nabla v + \frac{1}{2}(\partial_t
 f)^2\eta\otimes\eta + (\partial_t f)\mathrm{sym} (\eta\otimes\nabla
 v) \\ & + \mathrm{sym} \nabla w + (\partial_t g)\mathrm{sym}
 (\eta\otimes\nabla v) + (\partial_t h) \eta\otimes\eta + O\big(\frac{1}{\lambda}\big),
\end{split}
\end{equation*}
 yields  (\ref{ErrAlongEta}) provided that $f=-g$ and $\frac{1}{2}(\partial_t f)^2
+ \partial_t h = \phi $. The actual choice of $f,g$ and $h$ is given in
Lemma \ref{VW}.

\medskip

In the subsequent proof of Theorem \ref{Approx2}, we will aim at keeping the
frequencies $\lambda=\lambda_k$ minimal, facilitating a good numerical implementation. 
Three modifications with respect to proofs in \cite{LP} will be done
in order to improve the estimates on $\lambda_k$: 
\begin{itemize}
\item[(i)] all the global inequalities will be specified to pointwise estimates;
\item[(ii)] we will allow the error threshold parameter $\delta$ to depend on
$x$ rather than be constant; 
\item[(iii)] we will set distinct rates of convergence for the
errors $D_k$ and the approximations $v_k$ in Proposition \ref{Stage2}.
\end{itemize}

\bigskip

We first make the following simple observation:

\begin{lemma}\label{VW}
For a positive function $a\in \mathcal{C}^\infty(\bar{\Omega})$, the functions $V, W\in
{\mathcal{C}}^\infty(\bar\Omega\times \mathbb{R})$ defined by:
\begin{equation*}
V(x,t) \doteq \frac{a(x)}{\pi}\sin(2\pi t),\qquad
W(x,t) \doteq -\frac{a(x)^2}{4\pi}\sin(4\pi t),
\end{equation*}
are $1$-periodic in $t$, and they satisfy in $\bar\Omega\times \mathbb{R}$:
\begin{equation}\label{DiffRelation}
\frac{1}{2}(\partial_t V)^2 + \partial_t W = a^2,
\end{equation}
\begin{equation}\label{BoundsOscillations}
\begin{split} 
& |V|  \leq \frac{a}{\pi}, \quad  |\partial_t V| \leq 2a, \quad
|\nabla_x V| \leq \frac{|\nabla a|}{\pi}, \quad |\nabla_x^2 V| \leq
\frac{|\nabla^2 a|}{\pi}, \\
& |W|  \leq \frac{a^2}{4\pi}, \quad  |\partial_t W|\leq a^2, \quad
|\nabla_x W| \leq \frac{a |\nabla a|}{2\pi}.
\end{split} 
\end{equation}
\end{lemma}

We then obtain, consistently with (\ref{ErrAlongEta}):

\begin{proposition}\label{OneStep}
Let $v\in \mathcal{C}^\infty ( \bar{\Omega})$, $w\in \mathcal{C}^\infty
(\bar{\Omega},\R^2)$ and let $a\in \mathcal{C}^\infty (\bar{\Omega})$ be a
positive function. For a unit vector $\eta\in\R^2$ and a frequency $\lambda>0$,
define $v_\lambda\in \mathcal{C}^\infty ( \bar{\Omega},\R)$, 
$w_\lambda \in \mathcal{C}^\infty (\bar{\Omega},\R^2)$ through:
\begin{equation}\label{zzz}
\begin{split}
  v_\lambda(x) & \doteq v(x) + \frac{1}{\lambda}V(x,\lambda x\cdot\eta),\\
  w_\lambda(x) & \doteq w(x) -
  \frac{1}{\lambda}V(x,\lambda  x\cdot\eta)\nabla v(x) +
  \frac{1}{\lambda}W(x,\lambda x\cdot\eta)\eta.
 \end{split} 
\end{equation}
Then we have the following pointwise estimates, valid in $\bar\Omega$:
\begin{equation}\label{ErrEstEta}
\begin{split}  
\Big|\big(\frac{1}{2}\nabla v_\lambda\otimes\nabla v_\lambda + \mathrm{sym} \nabla
  w_\lambda\big) & -\big(\frac{1}{2}\nabla v\otimes\nabla v + \mathrm{sym} \nabla w + a^2
  \eta\otimes\eta\big) \Big|\\ 
& \quad \leq  \frac{1}{\lambda}\big(\frac{a |\nabla a|}{2\pi}
    +\frac{ a |\nabla^2 v|}{\pi} \big) + \frac{1}{2\lambda^2\pi^2}  |\nabla a|^2,
\end{split}
\end{equation}
\begin{align}
&  |v_\lambda - v| \leq \frac{a}{\lambda\pi}, \qquad
  |w_\lambda - w| \leq \frac{a}{\lambda\pi}\big(|\nabla  v|
    +\frac{a}{4}\big), \label{ErrVWlambda} \\
& |\nabla v_\lambda - \nabla v| \leq \frac{|\nabla a|}{\lambda\pi} +
2a, \label{ErrGrad}\\
&   \mbox{and } ~~~ |\nabla w_\lambda - \nabla w| \leq 2a|\nabla v| + a^2
   +\frac{1}{\lambda}\big( \frac{1}{\pi}|\nabla v||\nabla a| +
     \frac{a}{\pi}|\nabla^2v| +\frac{1}{2\pi} a|\nabla a| \big),\nonumber \\
&|\nabla^2 v_\lambda - \nabla^2 v| \leq \frac{|\nabla^2 a|}{\lambda\pi} + 4|\nabla a| + 4\lambda\pi a. \label{1.10}
\end{align}
\end{proposition}
\begin{proof}
To show \eqref{ErrVWlambda}, we use \eqref{BoundsOscillations} to estimate:
\begin{equation*}
|v_\lambda - v| = |\frac{1}{\lambda}V|\leq\frac{a}{\lambda\pi}, \qquad 
|w_\lambda - w| = \frac{1}{\lambda} |V\nabla v - W\eta| \leq
\frac{a}{\lambda\pi}\big(|\nabla v|+\frac{a}{4}\big).
\end{equation*}
Similarly, \eqref{ErrGrad} and (\ref{1.10}) follow, in view of
$|a\otimes b|= |a| \cdot |b|$:
\begin{equation*}
\begin{split}
&  |\nabla v_\lambda - \nabla v|  \leq \frac{1}{\lambda}|\nabla_xV| +
|\partial_tV| \leq \frac{|\nabla a|}{\lambda\pi} + 2a,\\ 
&  |\nabla w_\lambda - \nabla w|  = \Big|\frac{1}{\lambda}\nabla
    v\otimes\nabla_xV + (\partial_tV) \nabla v\otimes\eta  +
    \frac{1}{\lambda} V \nabla^2 v -  
  \frac{1}{\lambda} \eta\otimes\nabla_x W - (\partial_t W)\eta\otimes\eta\Big|\\
& \qquad\qquad \quad \; \leq 2a|\nabla v| + a^2 +\frac{1}{\lambda}\big(
    \frac{1}{\pi}|\nabla v||\nabla a| + \frac{a}{\pi}|\nabla^2v|
    +\frac{1}{2\pi} a|\nabla a| \big), \\ 
& |\nabla^2 v_\lambda - \nabla^2 v| \leq \frac{1}{\lambda} |\nabla^2_x
V| + 2|\nabla_x\partial_t V| + \lambda|\partial_t^2 V| 
\leq \frac{|\nabla^2 a|}{\lambda\pi} + 4|\nabla a| + 4\lambda\pi a.
\end{split} 
\end{equation*}
Lastly, we observe:
\begin{equation}\label{ErrCalc}
\begin{split}
\frac{1}{2}\nabla v_\lambda\otimes\nabla v_\lambda  + \mathrm{sym}
\nabla w_\lambda = & ~
  \frac{1}{2}\nabla v\otimes\nabla v + \frac{1}{2}(\partial_t
  V)^2\eta\otimes\eta + (\partial_t W) \eta\otimes \eta +  \mathrm{sym} \nabla w \\
&  + \frac{1}{\lambda}\big( (\partial_tV)
    \mathrm{sym} (\nabla_xV\otimes\eta)-V\nabla^2v+\mathrm{sym}
    (\nabla_xW\otimes\eta)\big) \\ & + 
  \frac{1}{2\lambda^2}\nabla_xV\otimes\nabla_xV.
\end{split}
\end{equation}
By\eqref{DiffRelation} and 
$(\partial_t V) (\nabla _xV) + \nabla_xW = \frac{1}{2\pi}\sin (4\pi t)
a\nabla a$, we arrive at:
\begin{equation*}
\begin{split}
\Big |\big(\frac{1}{2} \nabla v_\lambda\otimes&\nabla v_\lambda + \mathrm{sym} \nabla
  w_\lambda\big) - \big(\frac{1}{2}\nabla v\otimes\nabla v + \mathrm{sym} \nabla w + a^2
  \eta\otimes\eta\big) \Big| \\
& \leq\frac{1}{\lambda} \Big| 
    \mathrm{sym} \Big(\big((\partial_t V) (\nabla_xV) + \nabla_x
    W\big)\otimes\eta\Big) - V\nabla^2v \Big | + 
  \frac{1}{2\lambda^2} \Big |\nabla_xV \otimes\nabla_xV \Big| \\
& \leq  \frac{1}{\lambda}\big( \frac{a |\nabla a| }{2\pi} + \frac{a |\nabla^2 v|}{\pi} \big) +
  \frac{1}{2\lambda^2\pi^2}|\nabla a|^2,
\end{split}
\end{equation*}
concluding the proof of (\ref{ErrEstEta}).
\end{proof}

\begin{lemma} \label{CoeffFormula}
Matrices $\{\eta_k\otimes\eta_k\}_{k=1}^3$ form a basis of
$\R_{sym}^{2\times 2}$. If $B=[b_{ij}]_{i,j=1, 2} =
\sum_{k=1}^3 \phi_k \eta_k\otimes\eta_k$, then:
\begin{itemize}
\item[(i)] $ \phi_1 = b_{11}-\frac{1}{4} b_{22}$, $\phi_2 =
\frac{5}{8}(b_{22}+2b_{12})$, $\phi_3 = \frac{5}{8}(b_{22}-2b_{12})$.
\item[(ii)]  $\sum_{k=1}^3 \phi_k=
\mathrm{Tr} \,B$ and $|\phi_k|\leq \frac{5\sqrt{3}}{8}|B|$ for all
$k=1\ldots 3$.
\item[(iii)] If $\phi_k\geq d>0$ for all $k=1\ldots 3$, then $B\geq d
  \;\mathrm{Id}_2$.
\item[(iv)] Let $\tilde B = B + \alpha \,\mbox{diag}\big\{\frac{\sqrt{2}+9}{4},
  (\sqrt{2}+\frac{9}{5})\big\}$, with $\alpha\geq |B|$. Then:
  $|\tilde B|\leq 5.15\cdot \alpha$ and $\tilde B=\sum_{k=1}^3 
\tilde \phi_k\eta_k\otimes \eta_k$ with  $\tilde \phi_k\geq\frac{\alpha}{2}$.
\end{itemize}
\end{lemma}
\begin{proof}
The three indicated rank-one matrices are linearly independent by a
straightforward calculation. The formula in (i) follows directly as well and implies, in view of
the Cauchy-Schwartz inequa\-li\-ty: $|\phi_1|\leq
\frac{\sqrt{17}}{4}|B|$. In the same manner, we have:
$$|\phi_2|\leq \frac{5}{8}|b_{22} + b_{12} + b_{21}| \leq
\frac{5\sqrt{3}}{8}\big(b_{22}^2 + b_{12}^2 + b_{21}^2\big)^{1/2} \leq \frac{5\sqrt{3}}{8}|B|,$$ 
and likewise $|\phi_3|\leq \frac{5\sqrt{3}}{8}|B|$, so (ii) follows. For (iii), observe that
$\sum_{k=1}^3 \phi_k \eta_k\otimes\eta_k = \sum_{k=1}^3 (\phi_k-d)
\eta_k\otimes\eta_k + d\;\mbox{diag}\big\{\frac{7}{5}, \frac{8}{5}\big\} \geq 
d\;\mathrm{Id}_2$. In the setting of (iv), since  $b_{kk}\leq |B|$ for $k=1,2$ and
$b_{12}\leq\frac{\sqrt{2}}{2}|B|$, we get:
\begin{equation*}
\begin{split}
& \tilde \phi_1 = b_{11} + \frac{\sqrt{2} + 9}{4} |B| - \frac{1}{4}\big(b_{22} +
(\sqrt{2} + \frac{9}{5})\big) |B|\geq -|B| - \frac{1}{4}|B| + \frac{9}{5}|B|\geq\frac{1}{2}|B|,\\
& \tilde \phi_{2,3} = \frac{5}{8}\big(b_{22} + (\sqrt{2} + \frac{9}{5})|B|
\pm 2b_{12}\big)\geq \frac{1}{2}|B|.
\end{split}
\end{equation*}
Finally: $|\tilde B|\leq |B|+ \big(\big(\frac{\sqrt{2} + 9}{4}\big)^2 +
\big(\sqrt{2} +\frac{9}{5}\big)^2\big)^{1/2} |B|\leq 5.15\cdot |B|$.
\end{proof}

\begin{proposition}\label{Stage2}
Let $v\in \mathcal{C}^\infty ( \bar{\Omega}) $, $w\in \mathcal{C}^\infty
(\bar{\Omega},\R^2)$ and $A\in \mathcal{C}^\infty (\bar{\Omega},\R^{2\times2}_{sym})$ satisfy:
\begin{equation}\label{d}
D \doteq A - \big(\frac{1}{2}\nabla v\otimes\nabla v +\mathrm{sym}  \nabla w\big) =
\sum_{k=1}^3 \phi_k \eta_k\otimes\eta_k,\qquad \phi_k\geq d\quad \mbox{ in } \bar{\Omega},
\end{equation}
with some $\phi_k = \phi_k(x)$ and a constant $d>0$. Let $\xi>0$ be such that:
\begin{equation}\label{condixi}
\xi\leq  \min_{x\in\bar\Omega} |D(x)|.
\end{equation}
Fix $\epsilon>0$; then there exist $\tilde{v}\in \mathcal{C}^\infty (
\bar{\Omega})$, $\tilde{w}\in \mathcal{C}^\infty 
(\bar{\Omega},\R^2)$ and a constant $\tilde{d}>0$ such that:
\begin{align}
& \tilde D \doteq  A - \big(\frac{1}{2}\nabla \tilde v\otimes\nabla \tilde v +
\mathrm{sym}  \nabla \tilde w\big) = \sum_{k=1}^3 \tilde{\phi}_k
\eta_k\otimes\eta_k,\qquad \tilde{\phi}_k \geq \tilde d\quad \mbox{ in }
\bar{\Omega}, \label{tilded1}\\
& \|\tilde{D}\|_0\leq \frac{3}{4}\xi,\qquad \|\tilde{v}-v\|_0 \leq \epsilon, \qquad
\|\tilde{w}-w\|_0 \leq C\epsilon \big(\|\nabla
v\|_0+\|D\|_0^{1/2}\big), \label{tilded2} \\
& \|\nabla\tilde{v}-\nabla v\|_0 \leq C\|D\|_0^{1/2},\qquad
\|\nabla\tilde{w}-\nabla w\|_0 \leq C\big(\|D\|_0^{1/2}\|\nabla v\|_0+\|D\|_0\big), \label{tilded3}
\end{align}
where $C$ is a universal constant.
\end{proposition}

\begin{proof}
{\bf 1.}  We construct the intermediate fields $\{(v_k, w_k)\}_{k=1\ldots
3}$, from the given  $(v_0, w_0) \doteq (v, w)$  to the requested
$(\tilde{v}, \tilde{w})\doteq (v_3, w_3)$. Define smooth, positive functions
$\delta, a_k:\bar\Omega\to\mathbb{R}$:
$$\delta(x)|D(x)|=\frac{\xi}{2}, \qquad
a_k(x)^2=(1-\delta(x))\phi_k(x) \qquad \mbox{for } ~~k=1\ldots 3,$$ 
so that: $\delta D = D- \sum_{k=1}^3 a^2_k \eta_k\otimes\eta_k$. 
Clearly, (\ref{condixi}) guarantees the bound:
\begin{equation}\label{uno}
\delta\leq\frac{1}{2} \quad \mbox{ in } \; \bar\Omega.
\end{equation}
Given $(v_{k-1}, w_{k-1})\in \mathcal{C}^\infty(\bar\Omega,
\mathbb{R}^3)$, the successive corrections $v_k$ and $w_k$ are now constructed by applying Proposition
\ref{OneStep} to $v=v_{k-1}$, $w=w_{k-1}$, $a=a_k$, $\eta=\eta_k$ and an
appropriate $\lambda=\lambda_k\geq 1$ determined below in (\ref{tre})
and (\ref{assum2}). Observe that:
\begin{equation*}
\begin{split}
\tilde{D} & = D + \big( \frac{1}{2}\nabla v\otimes\nabla v + \mathrm{sym} \nabla
  w\big) - \big(\frac{1}{2}\nabla \tilde{v}\otimes\nabla \tilde{v} + \mathrm{sym} \nabla
  \tilde{w}\big)  \\
& = D- \sum_{k=1}^3 a^2_k \eta_k\otimes\eta_k  \\ & 
\qquad
- \sum_{k=1}^3 \Big(\big(\frac{1}{2}\nabla v_k\otimes\nabla v_k +
    \mathrm{sym} \nabla w_k\big) 
- \big(\frac{1}{2}\nabla v_{k-1}\otimes\nabla v_{k-1} +
    \mathrm{sym} \nabla w_{k-1}+a_k^2 \eta_k\otimes\eta_k \big)\Big)\\ 
&  =\delta D - \sum_{k=1}^3 B_k,
\end{split}
\end{equation*}
where (\ref{ErrEstEta}) yields the following pointwise bound on the
matrix-valued fields $\{B_k\}_{k=1}^3$:
\begin{equation}\label{help}
|B_k| \leq \frac{a_k|\nabla a_k|}{2\pi\lambda_k} + \frac{a_k|\nabla^2
  v_{k-1}|}{\pi\lambda_k} + \frac{|\nabla a_k|^2}{2\pi^2\lambda_k^2} \quad
\mbox{in } \bar\Omega.
\end{equation}

\medskip

{\bf 2.} To prove positivity of the decomposition in (\ref{tilded1}), we set:
$$\tilde d = \frac{\xi d}{4\|D\|_0} = \frac{d}{2} \min_{x\in
  \bar\Omega} \delta(x)$$
and use Lemma \ref{CoeffFormula} to:
$\sum_{k=1}^3 (\tilde \phi_k - \delta\phi_k)\eta_k\otimes \eta_k =
\tilde D - \delta D = -\sum_{k=1}^3 B_k$ in $\bar\Omega.$
Namely, by Lemma \ref{CoeffFormula} (ii) it follows that for
$k=1\ldots 3$ we have:
\begin{equation}\label{deco}
\tilde \phi_k \geq \delta\phi_k - \frac{5\sqrt{3}}{8} \big| \sum_{i=1}^3 B_i \big|\geq
\frac{\delta\phi_k}{2} \geq \frac{\delta d}{2}\geq \tilde d \qquad
\mbox{in } \bar\Omega, 
\end{equation}
where the second inequality above is valid when:
\begin{equation}\label{due}
\frac{5\sqrt{3}}{8}|B_i |\leq \frac{\delta\phi_k}{6} \qquad \mbox{in }
\bar\Omega, \quad \mbox{ for } i,k :1\ldots 3.
\end{equation}
Note that  the first estimate in (\ref{tilded2}) holds then as well,
again in view of Lemma \ref{CoeffFormula} (ii):
$$|\tilde D| \leq \delta |D| + \big|\sum_{i=1}^3 B_i \big| \leq \frac{\xi}{2} + \frac{4}{5\sqrt{3}}\delta\phi_i \leq
\frac{\xi}{2} + \frac{\delta}{2} |D| = \frac{3}{4}\xi \quad
\mbox{in } \bar\Omega.$$

For the validity of (\ref{due}), we choose $\{\lambda_k\}_{k=1}^3$ so that:
\begin{equation}\label{tre}
\begin{split}
 \frac{5\sqrt{3}}{8}\frac{a_i|\nabla a_i|}{2\pi\lambda_i} \leq
\frac{\delta\phi_k}{18} \quad & \mbox{and}\quad 
\frac{5\sqrt{3}}{8}\frac{a_i|\nabla^2 v_{i-1}|}{\pi\lambda_i} \leq
\frac{\delta\phi_k}{18} \\
& \mbox{and}\quad 
\frac{5\sqrt{3}}{8}\frac{|\nabla a_{i}|}{2\pi^2\lambda_i^2} \leq
\frac{\delta\phi_k}{18} \quad 
\mbox{in } \bar\Omega, \quad \mbox{ for } i,k:1\ldots 3.
\end{split}
\end{equation}

\medskip

{\bf 3.} Observe that, by Lemma \ref{CoeffFormula} and the definition of $a_k$, we get:
\begin{equation}\label{SumA2}
  \sum_{k=1}^3 a_k \leq \sqrt{3}\big(\sum_{k=1}^3
    a^2_k\big)^{1/2} = \sqrt{3}\big( (1-\delta) \mbox{Tr}\;D\big)^{1/2}
    \leq 3|D|^{1/2} \quad\mbox{in } \bar\Omega. 
\end{equation}
By (\ref{ErrVWlambda}), there follows the second inequality in (\ref{tilded2}):
\begin{equation*}
|\tilde v - v| \leq \sum_{k=1}^3 \frac{a_k}{\lambda_k\pi} \leq
\frac{3|D|^{1/2}}{\pi\displaystyle{\min_{k=1\ldots 3}}\{\lambda_k\}} \leq
\epsilon \quad\mbox{in } \bar\Omega,
\end{equation*}
if only we assume the following condition:
\begin{equation}\label{assum2}
\lambda_k \geq \frac{\|D\|_0^{1/2}}{\epsilon} \quad \mbox{for } k=1\ldots 3.
\end{equation}

Note that the third inequality in (\ref{tre}) and the bound on $\phi_k$ in terms of $|D|$ yield:
\begin{equation}\label{cinque}
\frac{|\nabla a_k|}{\pi\lambda_k}\leq
\Big(\frac{16}{5\sqrt{3}} \frac{\delta\phi_k}{18} \Big)^{1/2}
\leq \Big(\frac{\delta |D|}{9}\Big)^{1/2}\leq \frac{1}{3}|D|^{1/2} \quad\mbox{in } \bar\Omega.
\end{equation}
Thus, by (\ref{ErrGrad}) and (\ref{SumA2}) we conclude the first
bound in (\ref{tilded3}):
\begin{equation*}
|\nabla \tilde v - \nabla v| \leq \sum_{k=1}^3|\nabla v_k - \nabla v_{k-1}|\leq
\sum_{k=1}^3\big(\frac{|\nabla a_k|}{\lambda_k\pi} + 2a_k\big)\leq
7|D|^{1/2} \quad\mbox{in } \bar\Omega.
\end{equation*}
Similarly: $|\nabla v_k|\leq |\nabla v_0| + 7|D|_0^{1/2} $ for all $k=1\ldots 3$.
Using the second inequality in (\ref{ErrVWlambda}) together with 
(\ref{assum2}) and (\ref{SumA2}), we obtain the third bound in (\ref{tilded2}):
\begin{equation*}
\begin{split}
|\tilde w - w| & \leq \sum_{k=1}^3 \frac{a_k}{\lambda_k\pi}
\big(|\nabla v_{k-1}| + \frac{a_k}{4}\big) \leq \sum_{k=1}^3
\frac{a_k}{\lambda_k\pi} \big(|\nabla v| + 7 |D|^{1/2} +
\frac{a_k}{4}\big) \\ & \leq
\frac{|D|^{1/2}}{\displaystyle{\min_{k=1\ldots 3}}\{\lambda_k\}}\big(|\nabla v| + 8
|D|^{1/2} \big)\leq \epsilon \big(|\nabla v| + 8 |D|^{1/2} \big) \quad\mbox{in } \bar\Omega,
\end{split}
\end{equation*}
whereas (\ref{ErrGrad}) and (\ref{cinque}) are used for the final estimate in (\ref{tilded3}):  
\begin{equation*}
\begin{split}
|\nabla\tilde w & - \nabla w|  \leq \sum_{k=1}^3 \bigg( 2a_k |\nabla
v_{k-1}| + a_k^2 + \frac{2|\nabla v_{k-1}| |\nabla a_k| + a_k|\nabla
  a_|}{2\pi \lambda_k}  + \frac{2a_i
  |\nabla^2 v_{i-1}|}{2\pi\lambda_k}\bigg) \\ & \leq \sum_{k=1}^3
\bigg( 2a_k (|\nabla v| + 7|D|^{1/2}) +
a_k^2 + \frac{2|\nabla a_k| (|\nabla v| + 7 |D|^{1/2} )  +
  a_k|\nabla a_k|}{2\pi \lambda_k}  + \frac{a_k|\nabla^2 v_{k-1}|}{\pi\lambda_k}\bigg)  
\\ & \leq 7|D|^{1/2}\big(|\nabla v| + 7|D|^{1/2} \big) + 4
|D| \quad\mbox{in } \bar\Omega.
\end{split}
\end{equation*}
Above, the term $\frac{a_k|\nabla^2 v_{k-1}|}{\pi\lambda_k}$ has been bounded by
$\frac{1}{2}|D|_0$ in view of (\ref{tre}). 
\end{proof}

\bigskip

\noindent {\bf Proof of Theorem \ref{Approx2}.}

\smallskip

{\bf 1.} Assume first that (\ref{DefectBasis}) holds.
We will construct a sequence of smooth approximations $\{(v_k, w_k)\}_{k=0}^\infty$
which converge in $\mathcal{C}^1$ to the required solution in
(\ref{result}). Starting with $v_0$, $w_0$, we define recursively
$v_{k+1}$ and $w_{k+1}$ by applying Proposition \ref{Stage2} to
$v_{k}\in \mathcal{C}^\infty(\bar{\Omega})$ and $w_{k}\in
\mathcal{C}^\infty(\bar{\Omega},\R^2)$, where we denote the corresponding defect $D_k \doteq
A - (\frac{1}{2}\nabla v_k\otimes \nabla v_k + \mathrm{sym} \nabla
w_k)$, and request that the construction parameters $\epsilon_k, \xi_k>0$ satisfy: 
\begin{equation}\label{result2}
 \sum_{k=0}^\infty \epsilon_k\leq\beta \epsilon, \qquad
 \sum_{k=0}^\infty\xi_k^{1/2} \leq 1 \quad \mbox{ and } \quad \xi_k \leq
 \min_{x\in\bar\Omega} |D_k(x)| \quad \mbox{ for all } k\geq 0,
\end{equation}
for an appropriately small constant $\beta\leq 1$.
By (\ref{tilded1}),  each $D_k$ can be decomposed in the basis
$\{\xi_i\otimes \xi_i\}_{i=1}^3$ with strictly positive
coefficients. Further, (\ref{tilded2}) implies:
\begin{equation}\label{result3}
 \|v_k - v_0\|_0 \leq \sum_{i=0}^{k-1} \|v_{i+1}-v_{i}\|_0 \leq \sum_{i=0}^{k-1} \epsilon_i \leq \epsilon,
\end{equation}
while by (\ref{tilded3}) and (\ref{tilded2}) we obtain:
\begin{equation*}
 \|\nabla v_{k+m} - \nabla v_k\|_0\leq \sum_{i=k}^{k+m-1} \|\nabla
 v_{i+1} - \nabla v_{i}\|_0\leq C\sum_{i=k}^{k+m-1}\|D_{i}\|_0^{1/2}
 \leq C \sum_{i=k}^{k+m-1}\xi_{i-1}^{1/2}.
\end{equation*}
Consequently,  the sequence $\{v_k\}_{k=0}^\infty$ is Cauchy in
$\mathcal{C}^1(\bar\Omega)$ and thus it converges to some $v\in
\mathcal{C}^1(\bar\Omega)$. By (\ref{result3}), the first statement of 
(\ref{result}) follows. In particular, the norms  $\|\nabla v_k\|_0$
are uniformly bounded (by $C(\|\nabla v_0\|_0 +\|D_0\|_0 + 1)$, which allows to compute:
\begin{equation*}
\begin{split}
\| w_{k} - w_0\|_0 & \leq \sum_{i=0}^{k-1} \| w_{i+1} - w_{i}\|_0
\leq  C\sum_{i=0}^{k-1} \epsilon_{i} \big(\|\nabla 
 v_{i}\|_0 + \|D_{i}\|_0^{1/2}\big) \\ & \leq C \big(\|\nabla v_0\|_0
 +\|D_0\|_0  +\|D_0\|_0^{1/2} + 1\big)
\sum_{i=0}^{k-1}\epsilon_{i},\\
\|\nabla w_{k+m} - \nabla w_k\|_0 & \leq \sum_{i=k}^{k+m-1} \|\nabla
 w_{i+1} - \nabla w_{i}\|_0 \leq C \sum_{i=k}^{k+m-1}\|D_{i}\|_0^{1/2}
\big (\|\nabla v_{i}\|_0 + \|D_{i}\|_0^{1/2}\big) \\ & 
 \leq C \big(\|\nabla v_0\|_0 +\|D_0\|_0  +\|D_0\|_0^{1/2} + 1\big)
\sum_{i=k}^{k+m-1}\xi_{i-1}^{1/2},
\end{split}
\end{equation*}
where $C$ is a universal constant. Hence, $\{w_k\}_{k=0}^\infty$ is Cauchy and 
converges in $\mathcal{C}^1(\bar\Omega, \R^2)$ to some $w\in
\mathcal{C}^1(\bar\Omega, \mathbb{R}^2)$, satisfying:
\begin{equation*}
 \|A - \big(\frac{1}{2}\nabla v\otimes \nabla v + \mathrm{sym}\nabla
 w\big)\|_0 = \lim_{k\to\infty}\|D_k\|_0 = 0.
\end{equation*}
This proves the second statement in (\ref{result}), by (\ref{tilded2})
and (\ref{result2}). Also, it is clear that taking $\beta = C^{-1} \big(\|\nabla v_0\|_0
 +\|D_0\|_0  +\|D_0\|_0^{1/2} + 1\big)^{-1}$, there follows
 $\|w_k-w_0\|_0\leq \epsilon$ and so $\|w - w_0\|_0\leq \epsilon$.

\medskip

{\bf 2.} In the absence of the positivity of the decomposition
(\ref{DefectBasis}), we set $\tilde w_0(x,y) = w_0(x,y) -
\|D_0\|_0\big(\frac{\sqrt{2}+9}{4}x, (\sqrt{2}+\frac{9}{5})y\big)$. By
Lemma \ref{CoeffFormula} (iv), the decomposition  of the modified defect:
$$\tilde D_0\doteq A-\big( \frac{1}{2}\nabla v_0\otimes \nabla v_0 +
\Sym\nabla \tilde w_0\big) = D_0 +
\|D_0\|_0\cdot\mbox{diag}\Big\{\frac{\sqrt{2}+9}{4},
\sqrt{2}+\frac{9}{5}\Big\} = \sum_{k=1}^3 \phi_k\eta_k\otimes\eta_k,$$
obeys $\phi_k \geq \frac{\|D_0\|_0}{2}$ for $k=1\ldots 3$. Applying 
the first part of the proof to $v_0$, $\tilde w_0$ and $\tilde D_0$
yields $v\in\mathcal{C}^1(\bar\Omega)$ and
$w\in\mathcal{C}^1(\bar\Omega,\R^2)$ with the desired properties.
\endproof

\section{A numerical implementation of Theorem \ref{Approx2}}\label{Numerical1}

The purpose of this section is to obtain images of the first few
approximation steps in the convex integration construction of the
solutions to the Monge-Amp\`ere equation \eqref{MA}. We will consider
two case scenarios: (i) the right hand side $f(x,y) =
1$ and the subsolution $v_0(x,y) = x^2 -y^2$ is non-convex; (ii) the
right hand side $f(x,y) = -1$ and the subsolution $v_0(x,y) = x^2 +
y^2$ is strictly convex on the domain
$\Omega = (-0.5,0.5)\times(-0.5,0.5)$. We set the threshold $\epsilon =
0.1$ in seeking a solution 
$v$ with $\|v-v_0\|_0<0.1$ as in Theorem \ref{Approx2}. 

We were able to exhibit three consecutive corrugations.
Since the frequencies $\{\lambda_k\}$ quickly become very large, fine meshes needed to be
used, requiring a great computing power. On the other
hand, a priori estimates (\ref{tre}) and (\ref{assum2})
could be efficiently validated; indeed these estimates helped to
experimentally reduce the values of each $\lambda_k$.  

For the first approximation we sampled on a square grid with the
initial step size $0.001 $, followed by the step size of the order
$h\sim \frac{0.1}{\lambda_k}$. 
With this choice of $h$ we were so far able to perform two steps of
convex integration in both examples below. The derivatives on the mesh are:
\begin{equation*}
 \frac{\partial f}{\partial x}(x,\cdot) = \frac{1}{12h} 
 \big( f(x-2h,\cdot) - 8 f(x - h,\cdot) + 8 f(x + h, \cdot) - f(x + 2h,\cdot)
 \big)+O(h^5).
\end{equation*}
We assigned the values of the initial auxiliary variable
$w_0$ to make  the initial defect $D_0 = A- (\frac{1}{2}\nabla
v_0\otimes \nabla v_0 +\Sym \nabla w_0)$ diagonal, with the initial
positive definiteness constant $d=0.4\leq \phi_k$ for
all $k=1\ldots 3$. 
At each step, the defect $D$ was calculated explicitly and  then decomposed into $D = \sum_{k=1}^3
\phi_k\eta_k\otimes\eta_k$ using the formulas from Lemma \ref{CoeffFormula}. 
We set $\delta = 0.5$ and followed the indicated construction. 
By (\ref{tilded3}), our numerical implementation reduced the norm of defect to
$\frac{3}{4}$ of the previous defect norm every three convex integration
steps.

\smallskip

Define the coefficients:
\begin{equation*}
 \tilde a_k = \sqrt{\frac{\phi_k}{2}}
\end{equation*}
and the auxiliary matrices $\tilde{B}_k$ by:
\begin{equation*}
\tilde B_k = \big(\frac{1}{2}\nabla v_k\otimes\nabla v_k +
    \mathrm{sym} \nabla w_k\big) - \big(\frac{1}{2}\nabla v_{k-1}\otimes\nabla v_{k-1} +
    \mathrm{sym} \nabla w_{k-1}+\tilde a_k^2 \eta_k\otimes\eta_k \big).
\end{equation*}
From \eqref{help} we obtain bounds on $\tilde{B}_k$. To have
$\|\tilde{D}\|_0 \leq \frac{1}{2}\|D\|_0 + \sum_{k=1}^3\|B_k\|_0\leq
\frac{3}{4} \|D\|_0$ we request:
\begin{equation}\label{ConditionA}
 \|{\tilde B}_k\|_0\leq \frac{1}{12}\|D\|_0 \quad \mbox{ for all } k=1\ldots 3,
\end{equation}
while in order to get $\tilde\phi_k\geq \tilde d= \frac{d}{4}=0.1$, we proceed
as in (\ref{deco}) and introduce the condition:
\begin{equation}\label{ConditionB}
|{\tilde B}_k|<\frac{8}{15\sqrt{3}}\big(\frac{1}{2}\phi_i-0.1\big) \quad
\mbox{in } \bar\Omega, \quad \mbox{ for all } k,i=1\ldots 3.
\end{equation}

Conditions (\ref{ConditionA}) and (\ref{ConditionB}) were used to
determine the frequencies $\lambda_k$, where in estimating the three
terms in each $\tilde{B}_k$ according to (\ref{help}),  we bound $|\nabla
^2 v_{k-1}|$ using \eqref{1.10} with the previous choice of $\lambda_{k-1}$.
Once the three choices of $\lambda_k$ were made, we performed a
verification that the new $v$ and $w$ obtained with the described
above modification step remained indeed within
$\epsilon = 0.1$ error from the original fields. 

\smallskip

The values of $\lambda_1$ thus obtained were small enough to allow for 
the numerical executing the first step of convex integration. 
We defined the fields $v_1$ and $w_1$ as in Proposition \ref{OneStep}, 
calculated $\tilde B_1$, its decomposition into $\tilde B_1 =
\sum_{k=1}^3 \phi_k\eta_k\otimes\eta_k$, and verified the two conditions
 (\ref{ConditionA}) and (\ref{ConditionB}).
This operation has been repeated until the smallest possible $\lambda_1$ was found.
We then used this value and the resulting a priori estimates to obtain
the second frequency $\lambda_2$, small enough to apply the second step of convex integration. 
Thus we obtained the fields $v_2$ and $w_2$, and, by the same procedure described
above, the third optimal frequency $\lambda_3$ together with $v_3$
and $w_3$. The final step size $h$ needed to be reduced to
$0.0001$ in order to allow $10$ mesh points for each corrugation.  

\medskip

\begin{example}\label{Example_1}
We approximate $v_0(x,y) = x^2 - y^2$ with a solution $v$ to:
\begin{equation*}
  \Det \nabla^2 v  =  1.
\end{equation*}
Take $w_0(x,y) = (xy^2,x^2y)$ and $A(x,y) =
\big(5-\frac{x^2+y^2}{4} \big)\mbox{Id}_2$, so that: $-\cc A(x,y) = 1$.
Then:
\begin{equation*}
\frac{1}{2} \nabla v_0(x,y)\otimes\nabla v_0(x,y) + \mathrm{sym} \nabla w_0(x,y) =
\mbox{diag}\big\{ 2x^2+y^2, x^2+2y^2\big\}
\end{equation*}
and the corresponding defect is diagonal and positive definite:
\begin{equation*}
D(x,y) = \mbox{diag}\big\{
5-\frac{9x^2+5y^2}{4}, 5-\frac{5x^2+9y^2}{4}  \big\}
\end{equation*}
The function $v_0$ and its two subsequent corrugations are
shown in Figure \ref{Fig1}. In Figure \ref{Fig2} we show a detailed picture of the second corrugation; 
the red portion of the graph indicates the area on which we applied the
third corrugation  shown in Figure \ref{Fig3}.

\begin{figure}[h]
\centering
\begin{subfigure}{.33\textwidth}
  \centering
  \fbox{\includegraphics[height = 3.5cm,trim={7cm 0cm 7cm 2cm},clip]{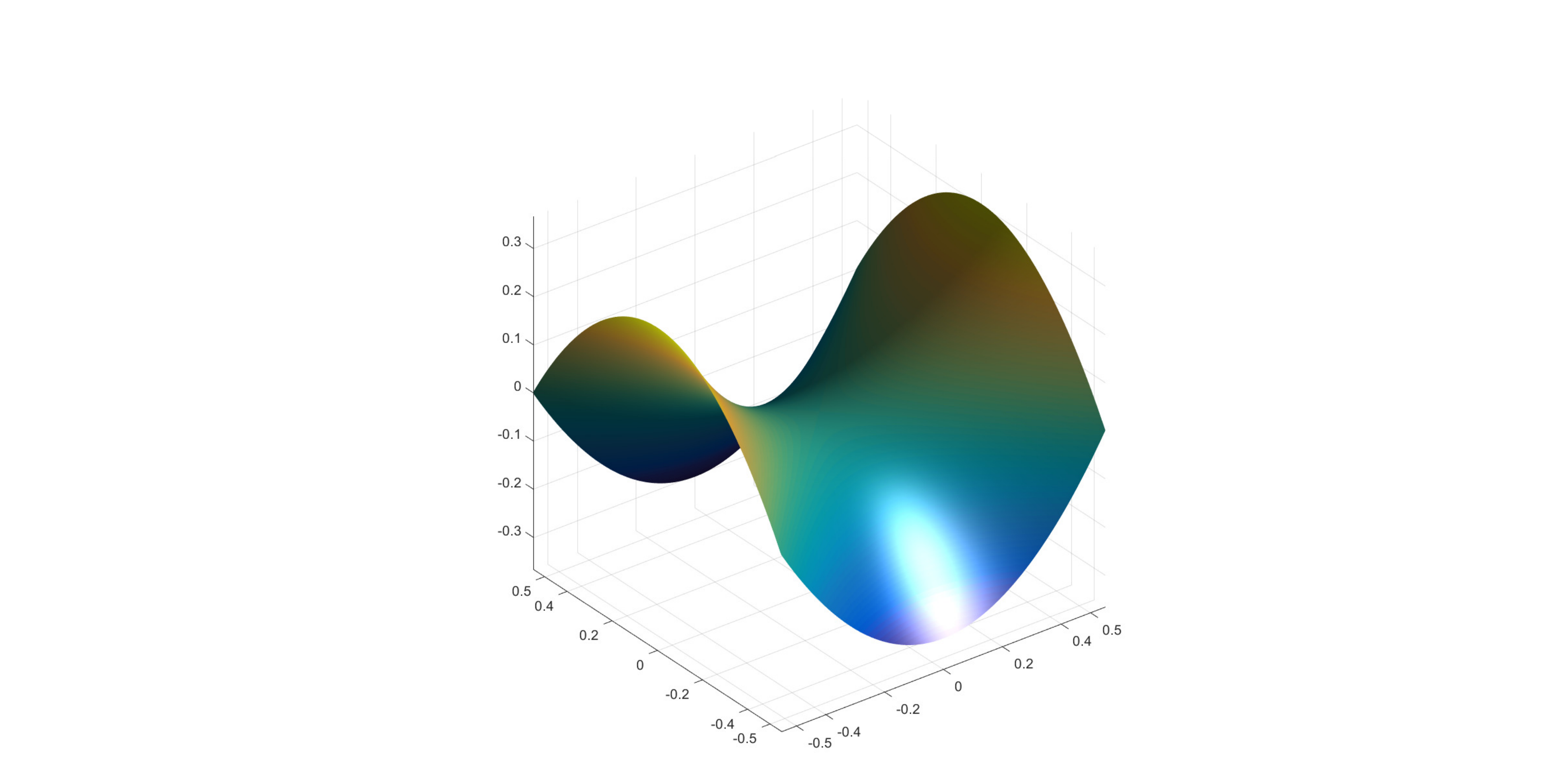}}
  \caption{The subsolution $v_0$ on $\bar\Omega$}
\end{subfigure}%
\begin{subfigure}{.33\textwidth}
  \centering
  \fbox{\includegraphics[height = 3.5cm,trim={7cm 1cm 7cm 1cm},clip]{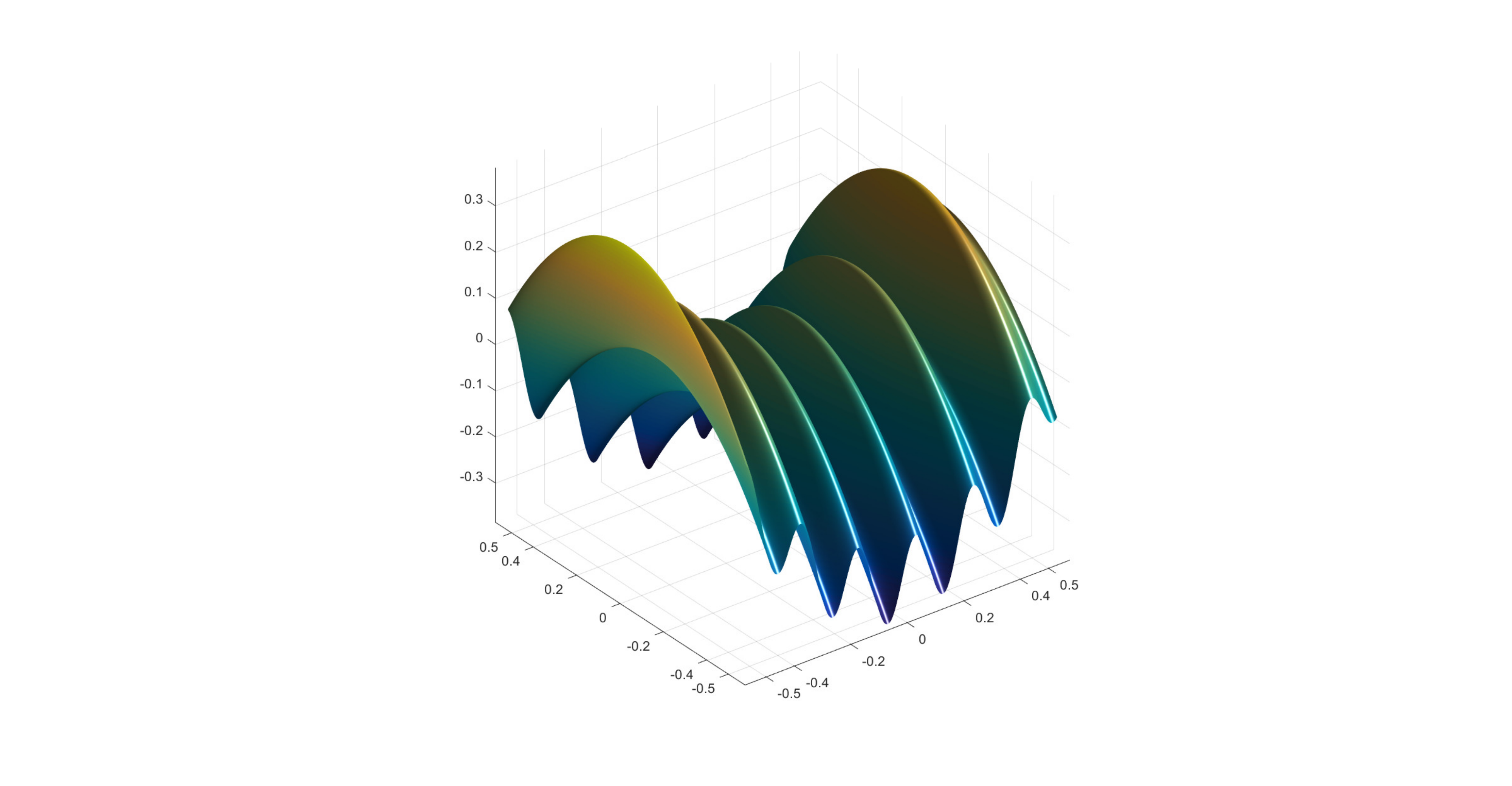}}
  \caption{One corrugation}
\end{subfigure}%
\begin{subfigure}{.33\textwidth}
  \centering
  \fbox{\includegraphics[height = 3.5cm,trim={7cm 1cm 7cm 1cm},clip]{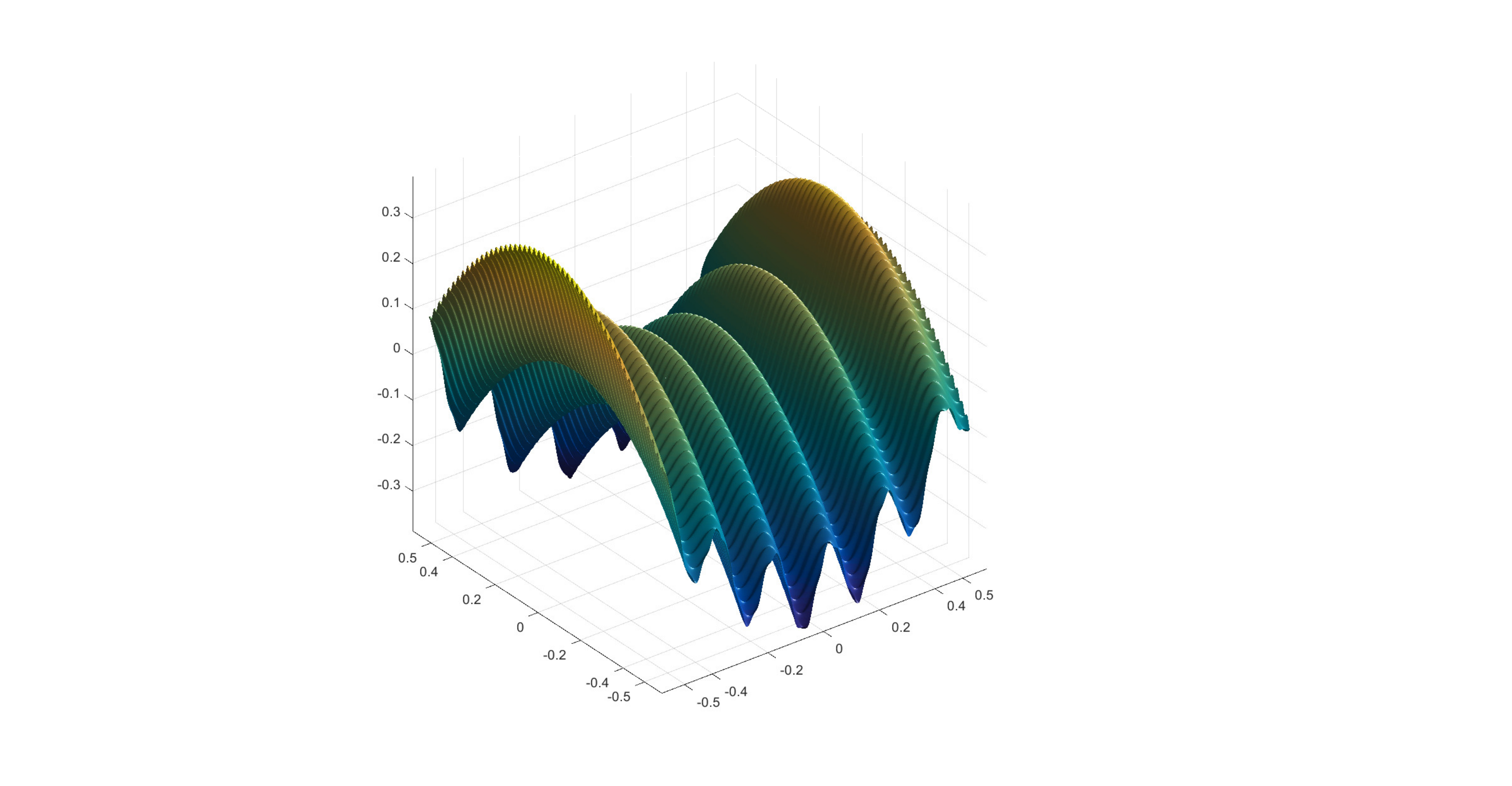}}
  \caption{Two corrugations}
\end{subfigure}
\caption{Construction in Example \ref{Example_1}}
\label{Fig1}
\end{figure}

\begin{figure}[h]
\fbox{\includegraphics[height=12cm,trim={8cm 1cm 8cm 1cm},clip]{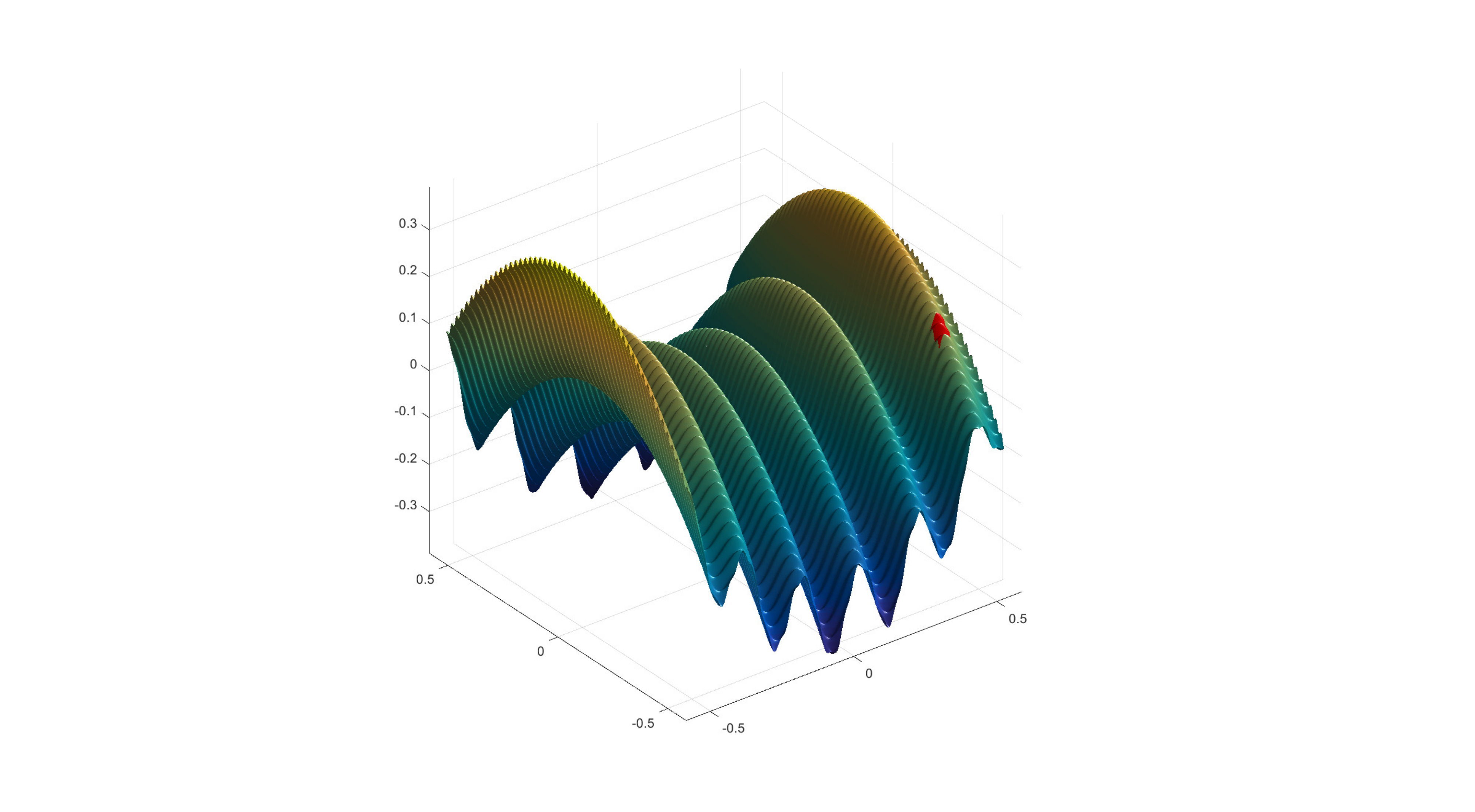}}
\caption{Two corrugations in Example \ref{Example_1}. The red detail
  is shown in Figure \ref{Fig3}}
\label{Fig2}
\end{figure}


\begin{figure}[h]
\fbox{\includegraphics[height=14cm,trim={9cm 1cm 9cm 1cm},clip]{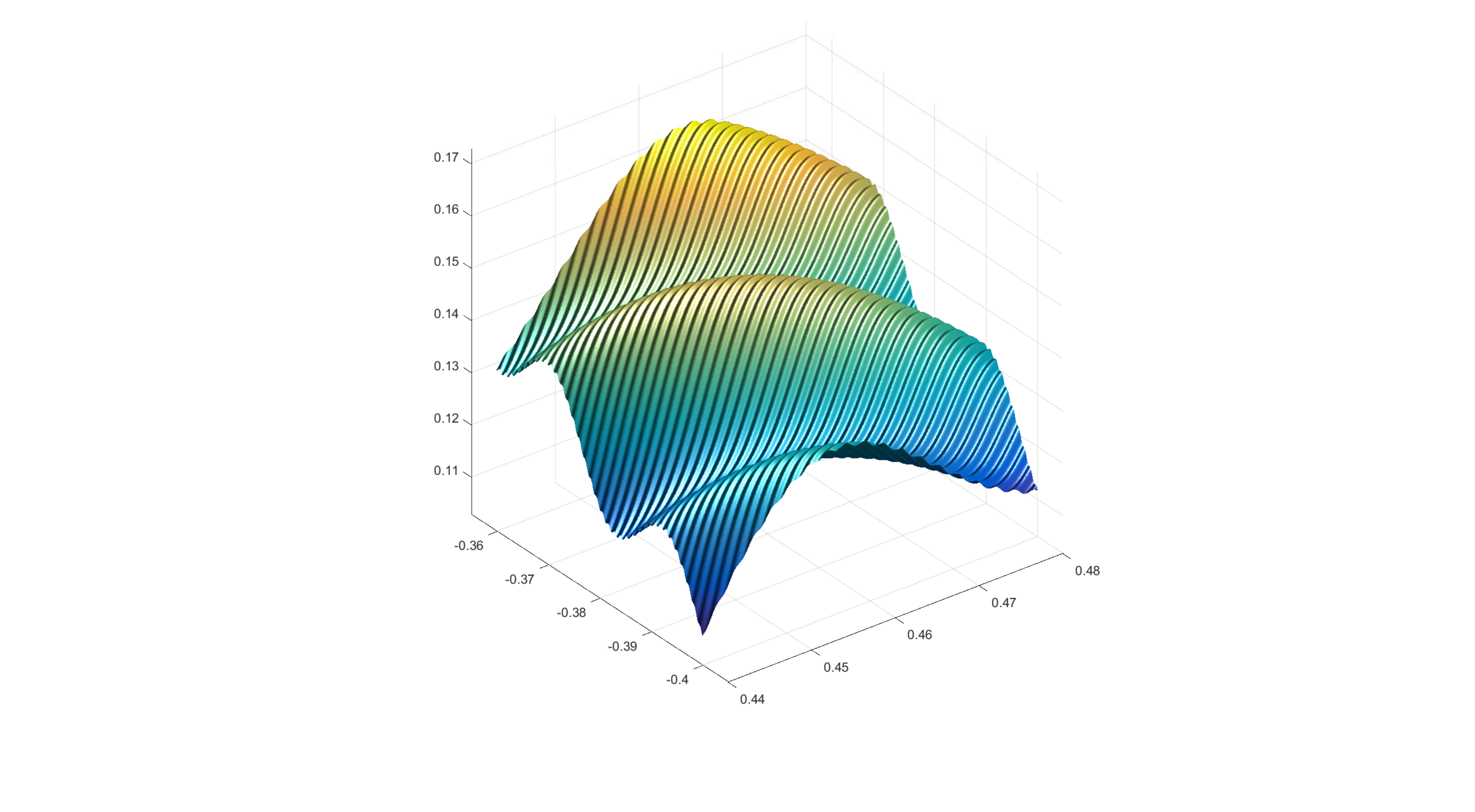}}
\caption{The detail of the three corrugations in Example \ref{Example_1}}
\label{Fig3}
\end{figure}
\end{example}

\begin{example}\label{Example_2}
We approximate $v_0(x,y) = x^2 + y^2$ with a solution $v$ to:
\begin{equation*}
 \Det \nabla^2 v  =  -1.
\end{equation*}
In this example we take $w_0(x,y) = (-xy^2,-x^2y)$ and
$A(x,y) =  \big(5+\frac{x^2+y^2}{4}\big) \mbox{Id}_2$, satisfying
$-\cc A(x,y) = -1$ and resulting in the diagonal, positive definite defect:
\begin{equation*}
D(x,y) =  \big\{
5-\frac{7x^2-5y^2}{4}, 5+\frac{5x^2-7y^2}{4} \big\}.
\end{equation*}
We plot three images starting from $v_0$ and subsequently adding the
first and second corrugations in Figure \ref{Fig11}.  As before, we provide a more detailed picture of the second
and third corrugations in Figures \ref{Fig22} and \ref{Fig33}.

\begin{figure}[h]
\centering
\begin{subfigure}{.33\textwidth}
  \centering
  \fbox{\includegraphics[height = 3.5cm,trim={7cm 1cm 7cm 1cm},clip]{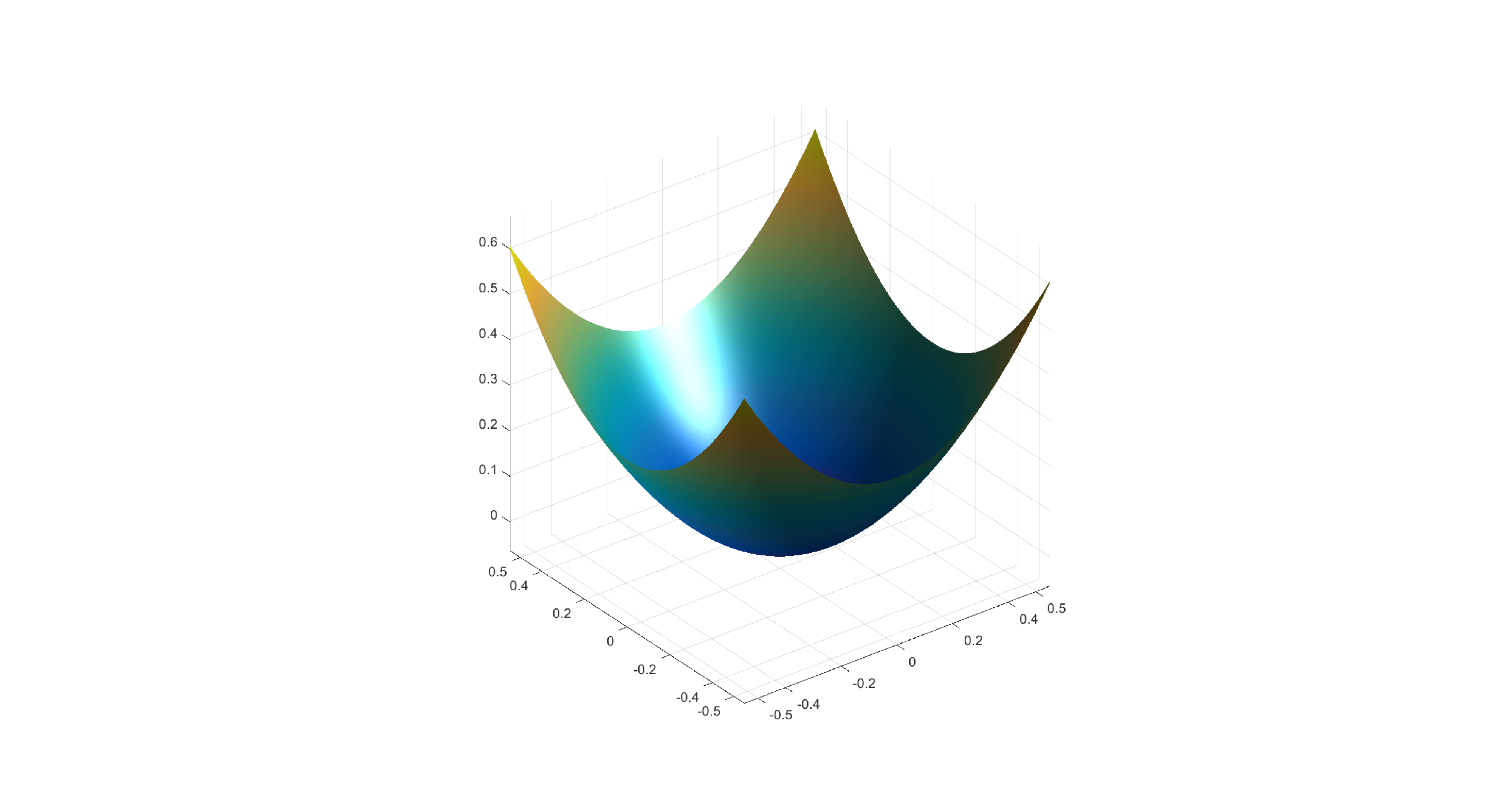}}
  \caption{Original function $v_0$}
\end{subfigure}%
\begin{subfigure}{.33\textwidth}
  \centering
  \fbox{\includegraphics[height = 3.5cm,trim={7cm 1cm 7cm 1cm},clip]{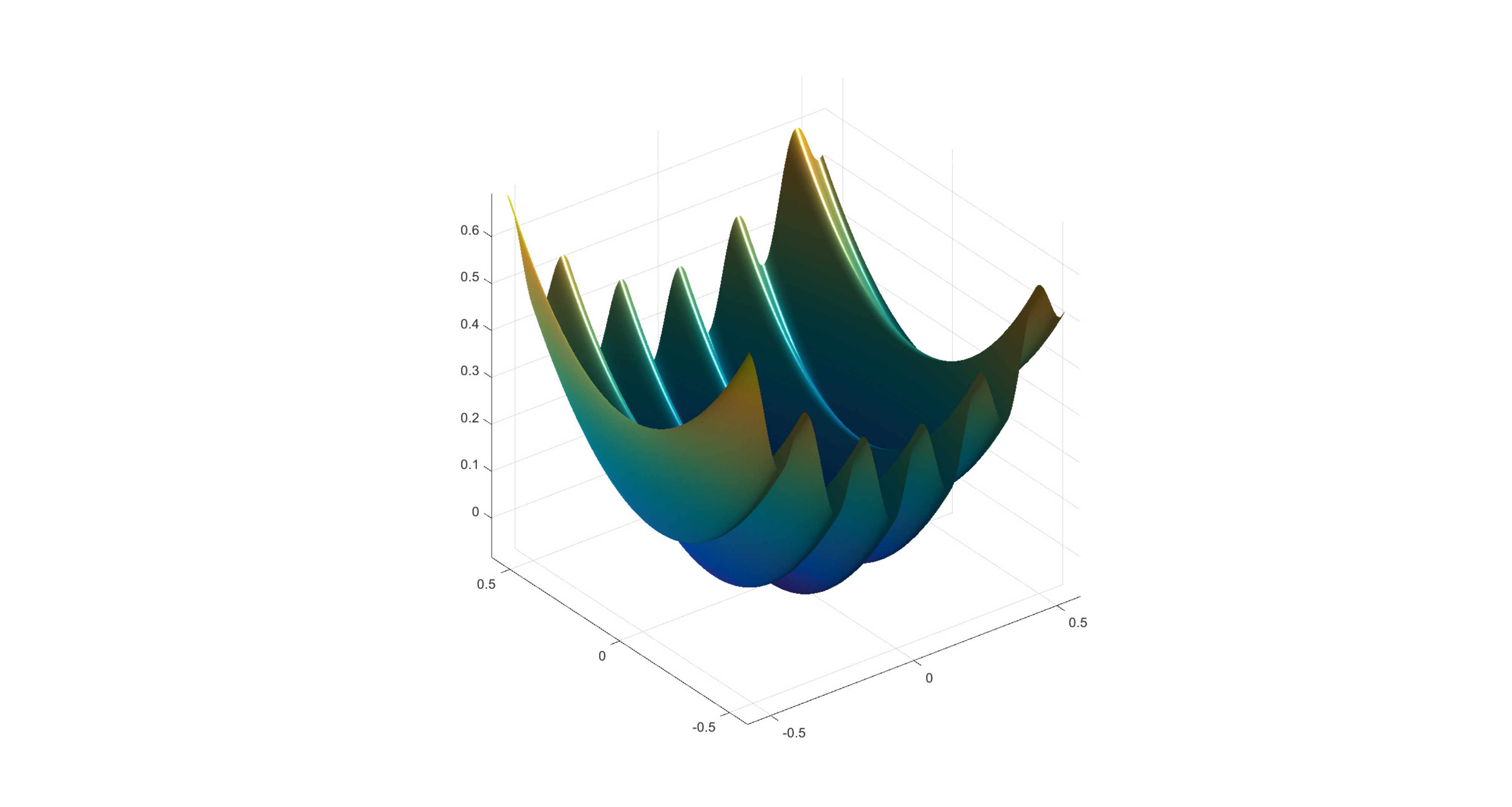}}
  \caption{One corrugation}
\end{subfigure}%
\begin{subfigure}{.33\textwidth}
  \centering
  \fbox{\includegraphics[height = 3.5cm,trim={7cm 1cm 7cm 1cm},clip]{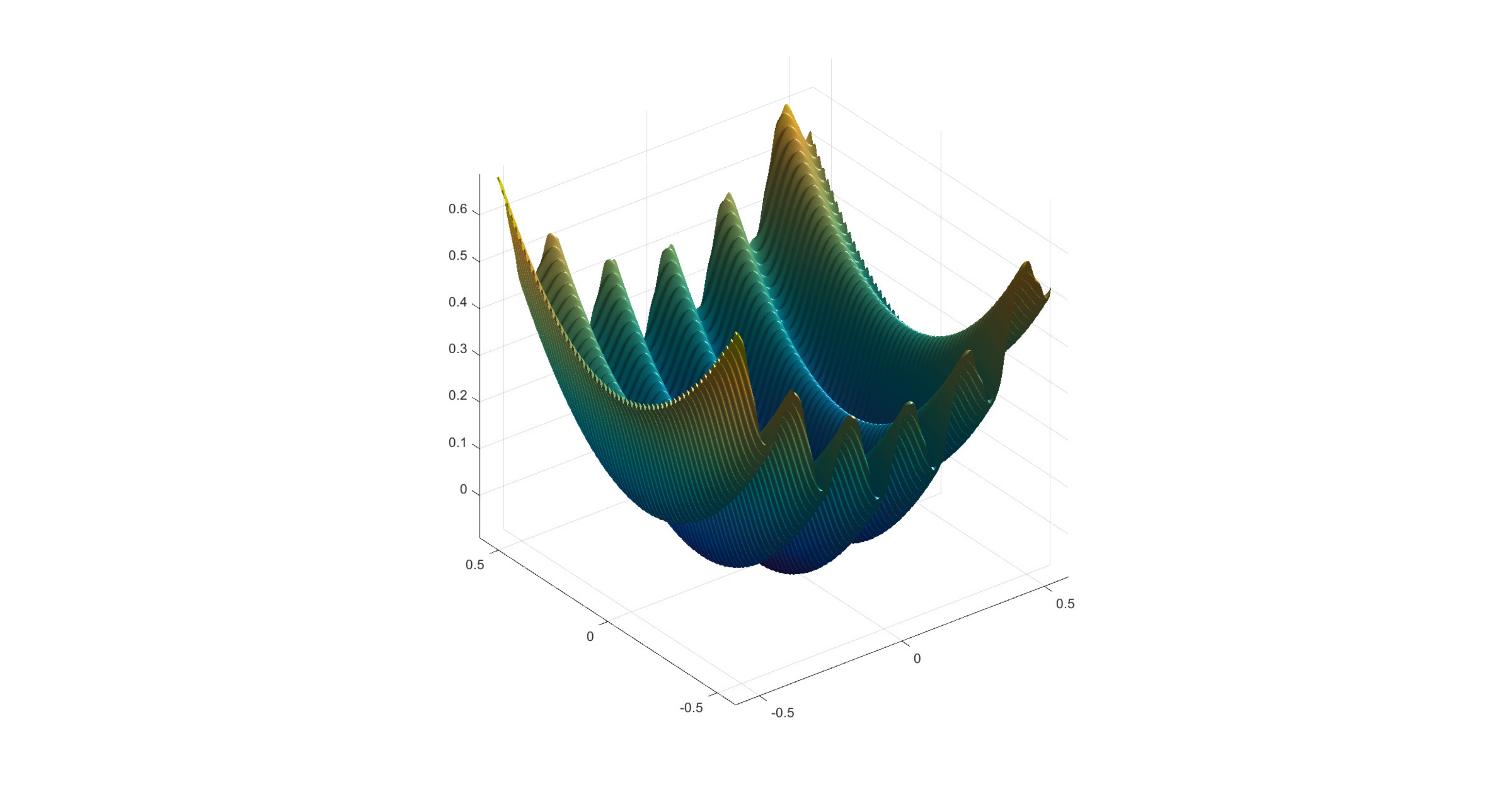}}
  \caption{Two corrugations}
\end{subfigure}
\caption{Construction in Example \ref{Example_2} }
\label{Fig11}
\end{figure}

\begin{figure}[h]
\fbox{\includegraphics[height=13cm,trim={8cm 1cm 8cm 1cm},clip]{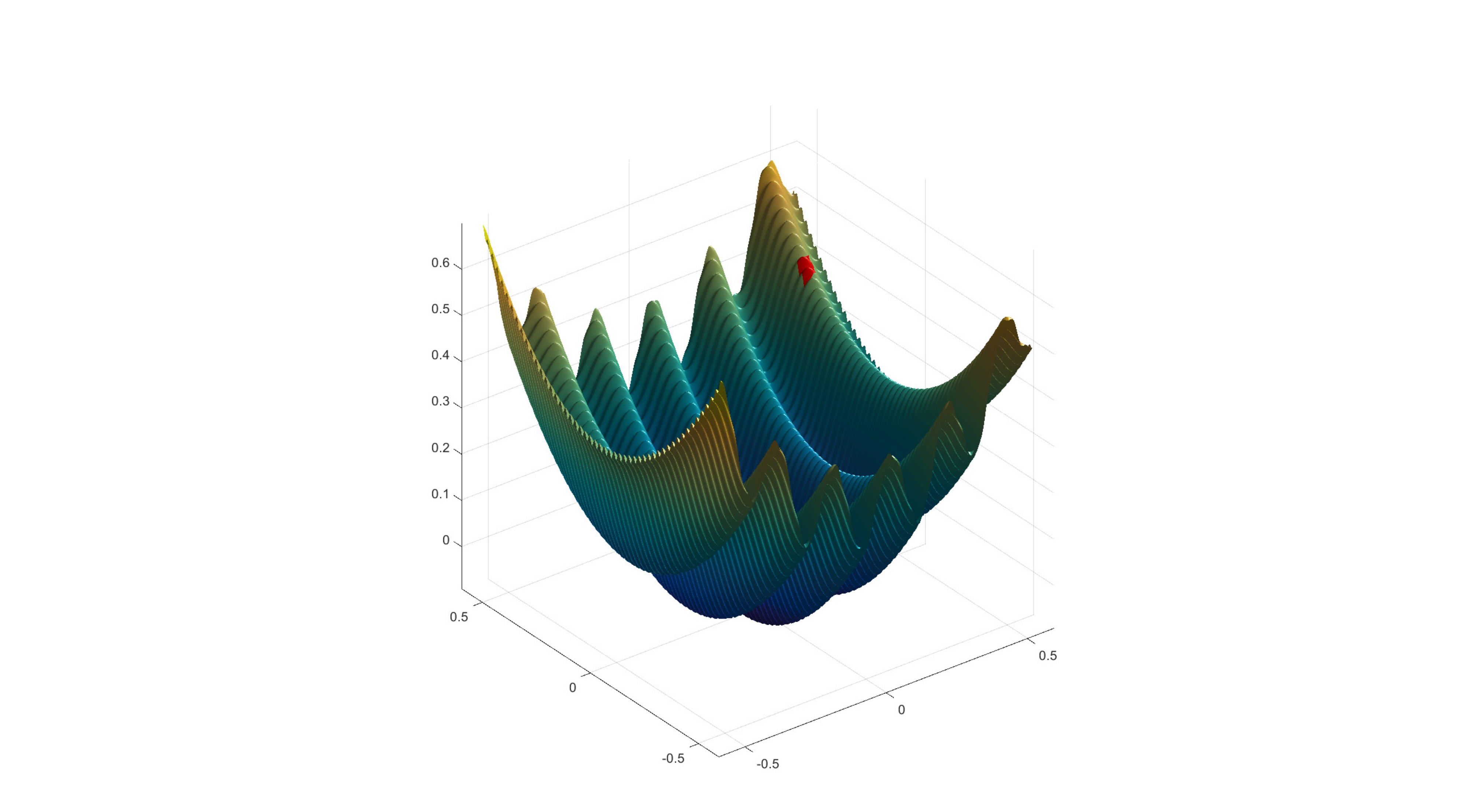}}
\caption{Two corrugations in Figure \ref{Fig11}. The red detail shown
  in Figure \ref{Fig33}}
\label{Fig22}
\end{figure}

\begin{figure}[h]
\fbox{\includegraphics[height=14cm,trim={9cm 1cm 9cm 1cm},clip]{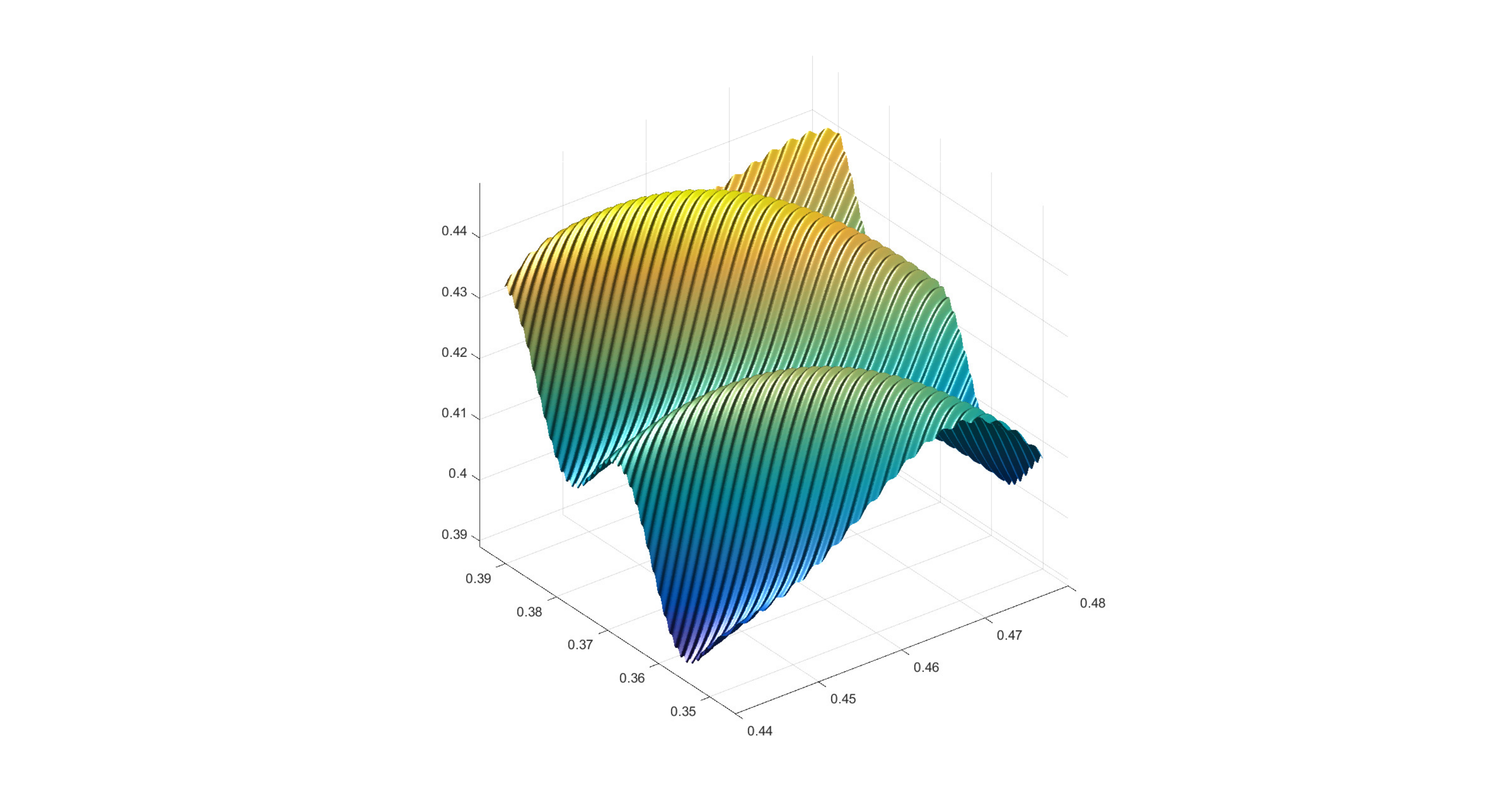}}
\caption{The detail of the three corrugations in Example \ref{Example_2}}
\label{Fig33}
\end{figure}
\end{example}

We conclude the discussion with a table listing some of the numerics
results and implementation choices. 
The values of $\{\lambda_k\}_{k=1}^3$ were obtained experimentally.
The values $\|v-v_0\|_0$ give an upper estimate of the uniform
distance of $v$ obtained through three steps of convex
integration from the initial subsolution $v_0$.
The value $(\|\tilde B_1\|_0+\|\tilde B_2\|_0)/\|D\|_0$ which does not
take the third corrugation into account, needs to be below
$2\cdot\frac{1}{12}=\frac{1}{6}$. 
The contribution of the last corrugation is guaranteed to be less than
$\frac{1}{12}\|D\|_0$ through the a priori estimates.  
Lastly,  $\min\phi_k$ are the minima of each of the coefficients in
$\bar\Omega$, in the defect computed after two corrugations. 
The a priori estimates once again guarantee that the third step will
not make these less than the required error value $0.1$.

\begin{center}
\begin{tabular}{ |c|c|c|c| } 
 \hline
   & Example \ref{Example_1} & Example \ref{Example_2} \\ 
 \hline
 $f(x,y)$ & 1 & -1  \\ 
 \hline
 $v_0(x,y)$ & $x^2-y^2$ & $x^2+y^2$ \\ 
 \hline
 $w_0(x,y)$ & $(xy^2,yx^2)$ & $(-xy^2,-yx^2)$  \\ 
 \hline
 $\lambda_1$ & 5 & 5  \\ 
 \hline
 $\lambda_2$ & 50 & 57  \\ 
 \hline
 $\lambda_3$ & 1000 & 1100  \\ 
 \hline
 $\|v-v_0\|_0$ & 0.0995 & 0.999 \\ 
 \hline
 $(\|\tilde B_1\|_0+\|\tilde B_2\|_0)/\|D\|_0$ & 0.1339 & 0.1246 \\ 
 \hline
 $\min\phi_1$ & 0.79 & 0.94 \\ 
 \hline
 $\min\phi_2$ & 1.14 & 1.29 \\ 
 \hline
 $\min\phi_3$ & 1.14 & 1.28 \\ 
 \hline
\end{tabular}
\end{center}

\section{The $\mathcal{C}^{1,\alpha}$ approximation}\label{sec4}

In this and the next section we give a proof of Theorem \ref{weakMA} using a constructive,
numerically implementable algorithm. The details of the
implementation and the resulting visualizations will be presented in
section \ref{num2_sec}. The construction follows the proof in
\cite{LP}, we however make some important modifications and compute all the
relevant constants explicitly. We begin with a choice of a standard
mollifier and some preliminary estimates. 

Let $\phi\in\mathcal{C}_c^\infty(B(0,1))$ be the following radially symmetric function:
\begin{equation}\label{Mollifier}
\varphi(x)= \begin{cases}
\displaystyle{\frac{1}{A}}\mbox{exp}\big(-\frac{1}{1-|x|^2}\big) & |x|\leq 1 \\
0 & |x|\geq 1, \end{cases}
\end{equation}
where $A=\int_{B_1(0)}\exp (-\frac{1}{1-|x|^2})~\mbox{d}x = \pi\cdot (1/e + Ei(-1))$ so that
$\int_{\R^2}\varphi =1$. The constant $A$, given in terms of the
exponential integral $Ei$, may be approximated to any degree of precision.
We will use $A\in [0.46, 0.47]$ in the estimates below, but when
implementing the algorithm numerically we will evaluate $A$ more precisely. 

\begin{lemma}\label{MolliLem}
Taking $\varphi$ as in (\ref{Mollifier}), we have:
\begin{equation}\label{norms}
  \|\varphi\|_{L^1(\R^2)} =1,\quad\|\nabla\varphi\|_{L^1(\R^2)} \leq
  3.1,\quad\|\nabla^2\varphi\|_{L^1(\R^2)} \leq
  15.9,\quad\|\nabla^3\varphi\|_{L^1(\R^2)} \leq 210. 
\end{equation}
Denote: $\varphi_l(x) = \frac{1}{l^2}\varphi\big(\frac{x}{l}\big)$ for all $l\in (0,1)$. 
Then, for every $f,g\in \mathcal{C}^0(\R^2)$ there holds:
\begin{equation} \label{MolIn1}
\|\nabla^{k+j}(f\ast\varphi_l)\|_0\leq \frac{1}{l^k}
\|\nabla^k\varphi\|_{L^1(\R^2)}\|\nabla^j f\|_0 \quad \mbox{ for
    all } ~ k,j\geq 0,
\end{equation}
\begin{equation}\label{MolIn2} 
\begin{split}
& \|f\ast\varphi_l-f\|_{0}\leq \frac{1}{2} l^2\|\nabla^2 f\|_0,\qquad
  \|\nabla(f\ast\varphi_l-f)\|_{0}\leq l\|\nabla^2 f\|_0 ,\\
&  \|\nabla^2 (f\ast\varphi_l-f)\|_{0}\leq 2\|\nabla^2 f\|_0.
\end{split}
\end{equation}
Moreover, for all $\alpha\in (0,1]$ there holds:
\begin{equation}\label{Molln3}
 \|f\ast\varphi_l-f\|_{0}\leq l^{\alpha}[f]_{\alpha}, \qquad
\|\nabla(f\ast\varphi_l)\|_{0}\leq (3.1) l^{\alpha-1}[f]_{\alpha},
\end{equation}
and further:
\begin{equation}\label{MolIn5}
\begin{split}
&  \|(fg)\ast\varphi_l-(f\ast\varphi_l)(g\ast\varphi_l)\|_{0}\leq
    2l^{2\alpha}[f]_{\alpha}[g]_{\alpha},\\ 
&  \|\nabla \big((fg)\ast\varphi_l-(f\ast\varphi_l)(g\ast\varphi_l)\big)\|_{0}\leq
  (9.3)l^{2\alpha-1}[f]_{\alpha}[g]_{\alpha},\\ 
&  \qquad \|\nabla^2\big((fg)\ast\varphi_l-(f\ast\varphi_l)(g\ast\varphi_l)\big)\|_{0}\leq
  (67)l^{2\alpha-2}[f]_{\alpha}[g]_{\alpha},\\ 
& \qquad \|\nabla^3\big((fg)\ast\varphi_l-(f\ast\varphi_l)(g\ast\varphi_l)\big)\|_{0}\leq
   (925.8) l^{2\alpha-3}[f]_{\alpha}[g]_{\alpha}. 
\end{split}
\end{equation}
All norms above are taken on the whole domain $\R^2$.
\end{lemma}
\begin{proof}
The estimates (\ref{norms}) follow by calculating the
indicated integrals in polar coordinates and then evaluating the
1-dimensional integrals numerically. The bound (\ref{MolIn1}) results from:
$\|\nabla^k\varphi_l\|_{L^1(\R^2)}=
l^{-k}\|\nabla^k\varphi\|_{L^1(\R^2)}$ for every $k\geq 0$, whereas to
get (\ref{MolIn2}) we use Taylor's expansion of $f$ at a given $x\in\R^2$.
Finally, writing $h=(fg)\ast\varphi_l-(f\ast\varphi_l)(g\ast\varphi_l)$, we observe as
in Lemma 2.1 \cite{CDS} that: 
\begin{equation*}
\begin{split}
& \|\nabla h\|_{0} \leq 3 l^{2\alpha} \|\nabla\varphi_l\|_{L^1(\R^2)} [f]_{\alpha}[g]_{\alpha}\\
& \|\nabla^2h\|_{0} \leq l^{2\alpha} \big(3\|\nabla^2\varphi_l\|_{L^1(\R^2)} +
2\|\nabla\varphi_l\|^2_{L^1(\R^2)} \big) [f]_{\alpha}[g]_{\alpha}\\
& \|\nabla^3h\|_{0} \leq l^{2\alpha} \big(3\|\nabla^3\varphi_l\|_{L^1(\R^2)} +
6\|\nabla^2\varphi_l\|^2_{L^1(\R^2)} \|\nabla\varphi_l\|_{L^1(\R^2)} \big) [f]_{\alpha}[g]_{\alpha},
\end{split}
\end{equation*}
which proves \eqref{MolIn5} in view of (\ref{norms}).
\end{proof}

The next result is a modification of Proposition \ref{OneStep} so we omit its proof.

\begin{proposition}\label{OneStepMod}
Given $v\in \mathcal{C}^3(\bar{\Omega})$, $w\in \mathcal{C}^2(\bar{\Omega},\R^2)$,
a nonnegative function $a\in \mathcal{C}^3(\bar{\Omega})$ and a unit
vector $\eta\in\R^2$, let $\delta, l\in(0,1)$ be two parameter constants satisfying:
\begin{equation}\label{DeltaLConditions}
 \|\nabla^m a\|_0\leq\frac{\delta}{l^m} \quad \mbox{ for } ~m =
 0\ldots 3\quad\mbox{and} \quad \|\nabla^{m+1} v\|_0\leq\frac{\delta}{l^m}
 \quad \mbox{ for } ~m = 1,2.  
\end{equation}
Then for any frequency $\lambda \geq 1/l$, the approximations $v_\lambda\in
\mathcal{C}^3(\bar{\Omega})$ and  $w_\lambda \in
\mathcal{C}^2(\bar{\Omega},\R^2)$  defined in (\ref{zzz}), satisfy:
\begin{equation}\label{OneStepMod1}
  \Big|\big(\frac{1}{2}\nabla v_\lambda\otimes\nabla v_\lambda + \mathrm{sym} \nabla
  w_\lambda\big)  -\big(\frac{1}{2}\nabla v\otimes\nabla v + \mathrm{sym} \nabla w + a^2 \eta\otimes\eta\big) \Big|
  \leq  \frac{\delta^2}{\lambda l},
\end{equation}
\begin{equation}\label{OneStepMod2}
\begin{split}  
&  \|v_\lambda-v\|_0 \leq (0.4)\frac{\delta}{\lambda}, \quad 
\|\nabla(v_\lambda-v)\|_0 \leq (2.4)\delta, \\ & \|\nabla^2(v_\lambda-v)\|_0 \leq (16.9)\delta\lambda, \quad 
\|\nabla^3(v_\lambda-v)\|_0 \leq (123)\delta\lambda^2\\
&  \qquad \|w_\lambda-w\|_0 \leq (0.4)\frac{\delta}{\lambda}(1+\|\nabla v\|_0), \quad 
\|\nabla(w_\lambda-w)\|_0 \leq (2.4)\delta (1+\|\nabla v\|_0), \\ & 
\qquad \|\nabla^2(w_\lambda-w)\|_0 \leq (21.9)\delta\lambda (1+\|\nabla v\|_0).
\end{split}
\end{equation}
\end{proposition}

The following is the ``stage'' of the H\"older approximation
construction, consisting of iterating three convex integration steps
detailed in Proposition \ref{Stage2}, with an additional mollification at
each step in order to control the second derivative norms. 
The statement and the proof are similar to Proposition 5.2 in
\cite{LP}, but we avoid the extension argument and consequently make
the universal constants explicit to allow for a numerical
implementation. 

\begin{proposition}[Proposition 5.2 \cite{LP}]\label{Stage3}
For an open  bounded domain $\Omega \subset \R^2$, we denote:
$\Omega_r\doteq\Omega + B_r(0)$ for some $r\in (0,1)$. 
Given three functions $v\in \mathcal{C}^2(\bar{\Omega}_r)$, $w\in
\mathcal{C}^2(\bar{\Omega}_r,\R^2)$ and $A\in
\mathcal{C}^{0,\beta}(\bar{\Omega}_r,\R^{2\times 2}_{sym})$ of H\"older
regularity $\beta\in(0,1)$, assume that for $\delta_0< (5.4)\cdot 10^{-16}$ we have:
\begin{equation}\label{DefectSmall}
  D \doteq A-\big(\frac{1}{2}\nabla v\otimes\nabla v + \mathrm{sym} \nabla w\big), 
\quad 0<\|D\|_{\mathcal{C}^0(\bar\Omega_r)}\leq\delta_0.   
\end{equation}
Then, for every two constants $M, \sigma$ which satisfy:
\begin{equation}\label{MSigmaConditions}
  M>\max\Big\{\frac{\|D\|_{\mathcal{C}^0(\bar\Omega_r)}^{1/2}}{r},
\|\nabla^2 v\|_{\mathcal{C}^0(\bar\Omega_r)},\|\nabla^2
  w\|_{\mathcal{C}^0(\bar\Omega_r)},1\Big\} \quad\mbox{and} \quad\sigma>1,   
\end{equation}
there exist $\tilde{v}\in \mathcal{C}^2(\bar{\Omega})$ and
$\tilde{w}\in \mathcal{C}^2(\bar{\Omega},\R^2)$ such that with:
$\tilde D \doteq A-\big(\frac{1}{2}\nabla \tilde{v}\otimes\nabla
\tilde{v} + \mathrm{sym} \nabla \tilde{w}\big)$ we have:
\begin{equation}\label{DefectStage2}
  \|\tilde{D}\|_0 \leq
  \frac{\|A\|_{\mathcal{C}^{0,\beta}(\bar\Omega_r)}}{M^\beta}\|D\|_{\mathcal{C}^0(\bar\Omega_r)}^{\beta/2}
  + \frac{(1.9)10^{15}}{\sigma}\|D\|_{\mathcal{C}^0(\bar\Omega_r)}, 
\end{equation}
\begin{equation}\label{Norm0Est}
\begin{split}
\|\tilde{v}-&v\|_0  \leq  \frac{(1.8) 10^7}{M}\|D\|_{\mathcal{C}^0(\bar\Omega_r)}\\
& \mbox{and } \quad \|\tilde{w}-w\|_0\leq \Big(\frac{(1.8) 10^7}{M}+
(12.6)\mathrm{diam}(\Omega_r)\Big)\|D\|_{\mathcal{C}^0(\bar\Omega_r)}
(1+\|\nabla v_0\|_{\mathcal{C}^0(\bar\Omega_r)}),
\end{split}
\end{equation}
\begin{equation}\label{Norm1Est}
 \begin{split}
&\|\nabla(\tilde{v}-v)\|_0\leq
(1.1) 10^8\|D\|_{\mathcal{C}^0(\bar\Omega_r)}^{1/2} \\ & \qquad\qquad \qquad
\qquad \qquad \mbox{and}\quad
 \|\nabla(\tilde{w}-w)\|_0\leq (1.1)10^8(1+\|\nabla
 v\|_{\mathcal{C}^0(\bar\Omega_r)})\|D\|_{\mathcal{C}^0(\bar\Omega_r)}^{1/2},  
\end{split}
\end{equation}
\begin{equation}\label{Norm2Est}
  \|\nabla^2\tilde{v}\|_0\leq (7.3)10^8 M\sigma^3 \quad \mbox{and}\quad
  \|\nabla^2\tilde{w}\|_0\leq (9.5)10^8 (1+\|\nabla v\|_{\mathcal{C}^0(\bar\Omega_r)})M\sigma^3. 
\end{equation}
The norms $\|\cdot\|_0$ above signify $\|\cdot\|_{\mathcal{C}^0(\bar\Omega)}$.
\end{proposition}
\begin{proof}
The proof proceeds in three parts: mollification of $v, w, A$ to control higher
derivatives, modification of $w$ to ensure the positive
decomposition of the defect in the basis
$\{\eta_k\otimes\eta_k\}_{k=1}^3$, and application of three consecutive steps of
convex integration to reduce the defect.

\smallskip

{\bf 1. Mollification.}
For all $l<r$, define the following functions on $\bar\Omega$ through
mollification with the standard kernel $\varphi$ as in \eqref{Mollifier}:
\begin{equation*}
  \mathfrak{v} \doteq  v \ast \varphi_l, \quad \mathfrak{w} \doteq
  w \ast \varphi_l, \quad \mathfrak{A} \doteq A \ast \varphi_l, \quad \mbox{where}
  \quad l \doteq \frac{\|D\|_{\mathcal{C}^0(\bar\Omega_r)}^{1/2}}{M}<r<1. 
\end{equation*}
Denote: $\mathfrak{D}\doteq\mathfrak{A}-\left(\frac{1}{2}\nabla
  \mathfrak{v}\otimes\nabla \mathfrak{v} + \mathrm{sym} \nabla
  \mathfrak{w}\right)$.  
We use Lemma \ref{MolliLem} and assumption (\ref{MSigmaConditions}) to
obtain the uniform error bounds below, where the relevant norms of
quantities $\mathfrak{v}, \mathfrak{w}, \mathfrak{A}, \mathfrak{D}$
are taken on $\bar\Omega$, while other norms are taken on the superset $\bar\Omega_r$:
\begin{equation}\label{MolliResult}
\begin{split}
&  \|\mathfrak{v} -  v\|_0, ~\|\mathfrak{w} - w\|_0  \leq
 \frac{l}{2}\|D\|_{\mathcal{C}^0(\bar\Omega_r)}^{1/2}, \quad 
\|\nabla(\mathfrak{v} -  v)\|_0, ~ \|\nabla(\mathfrak{w} - 
w)\|_0 \leq \|D\|_{\mathcal{C}^0(\bar\Omega_r)}^{1/2},\\
&  \|\mathfrak{A} - A\|_0  \leq l^\beta[A]_{\beta, \bar\Omega_r},\\
&  \|\nabla ^m \mathfrak{D}\|_0 \leq\|\nabla^m(
D\ast\varphi_l)\|_0 + \frac{1}{2}\|\nabla^m\big((\nabla
  v\ast\varphi_l)\otimes(\nabla v\ast\varphi_l) - (\nabla
  v\otimes\nabla v)\ast\varphi_l\big)\|_0.
\end{split}
\end{equation}
In view of (\ref{norms}), (\ref{MolIn5}) and (\ref{MSigmaConditions}),
the last bound above is specified to:
\begin{equation}\label{NormDMol}
\begin{split} 
& \|\mathfrak{D}\|_0\leq \| D\|_0 + l^2[\nabla v]_1^2 \leq 2\|D\|_{\mathcal{C}^0(\bar\Omega_r)},\\ 
& \|\nabla \mathfrak{D}\|_0\leq \frac{3.1}{l }\|D\|_0 + (4.7)
l[\nabla  v]_1^2 \leq (7.8)\frac{1}{l} \|D\|_{\mathcal{C}^0(\bar\Omega_r)},\\ 
& \|\nabla^2 \mathfrak{D}\|_0\leq \frac{15.9}{l^2}\| D\|_0 + (33.5)
[\nabla  v]_1^2 \leq (49.4)\frac{1}{l^2} \|D\|_{\mathcal{C}^0(\bar\Omega_r)},\\ 
& \|\nabla^3 \mathfrak{D}\|_0\leq \frac{210}{l^3}\| D\|_0 + (462.9)
\frac{1}{l}[\nabla v]_1^2 \leq (672.9)\frac{1}{l^3} \|D\|_{\mathcal{C}^0(\bar\Omega_r)}.
\end{split}
\end{equation}
We also get, by \eqref{MolIn1}:
\begin{equation}\label{controlvnorms}
\begin{split}
& \|\nabla^2 \mathfrak{v}\|_0 \leq \|\nabla^2 v\|_{\mathcal{C}^0(\bar\Omega_r)} \leq
\frac{1}{l}\|D\|_{\mathcal{C}^0(\bar\Omega_r)}^{1/2}, \\ & \|\nabla^3 \mathfrak{v}\|_0
\leq \frac{1}{l}\|\nabla\varphi\|_{L^1(\R^2)} \|\nabla^2 v\|_{\mathcal{C}^0(\bar\Omega_r)} \leq
(3.1)\frac{1}{l^2}\|D\|_{\mathcal{C}^0(\bar\Omega_r)}^{1/2}.
\end{split}
\end{equation}
Finally,  Lemma \ref{MolliLem} yields:
\begin{equation}\label{controlwnorms}
  \|\nabla^2 \mathfrak{w}\|_{0}\leq \|\nabla^2 w\|_{\mathcal{C}^0(\bar\Omega_r)}\leq M.  
\end{equation}
 
\smallskip

{\bf 2. Modification and decomposition.}
To ensure that the deficit may be decomposed with positive coefficients, we define:
\begin{equation*}
\tilde{\mathfrak{w}} \doteq \mathfrak{w} - (\|D\|_{\mathcal{C}^0(\bar\Omega_r)}+\|\mathfrak{D}\|_0)
\Big( \frac{\sqrt{2}+9}{4} x, (\sqrt{2}+\frac{9}{5}) y\Big),
 \qquad \tilde{\mathfrak{D}} \doteq \mathfrak{A} - \big(\frac{1}{2}\nabla
    \mathfrak{v}\otimes\nabla \mathfrak{v} + \mathrm{sym} \nabla
  \tilde{\mathfrak{w}}\big).  
\end{equation*}
By \eqref{MolliResult} there follow the bounds:
\begin{equation}\label{wModEst}
\begin{split} 
& \|\nabla(\tilde{\mathfrak{w}} -\mathfrak{w})\|_0  \leq (4.2) (\|
D\|_{\mathcal{C}^0(\bar\Omega_r)}+\|\mathfrak{D}\|_0)\leq (12.6)\| D\|_{\mathcal{C}^0(\bar\Omega_r)}, \\
& \|\nabla^2(\tilde{\mathfrak{w}} - \mathfrak{w})\|_0 = 0.
\end{split}
\end{equation}
Note that $\tilde{\mathfrak{D}} -\mathfrak{D} = \big(
\|D\|_{\mathcal{C}^0(\bar\Omega_r)} + \|\mathfrak{D}\|_0\big)
\mbox{diag}\big(\frac{\sqrt{2}+9}{4}, \sqrt{2}+\frac{9}{5}\big)$ and therefore:
$\nabla^m \tilde{\mathfrak{D}} = \nabla^m \mathfrak{D}$ for all $m\geq 1$.
Further, this construction guarantees that in the decomposition
$\tilde{\mathfrak{D}} = \sum_{k=1}^3\phi_k\eta_k\otimes\eta_k$ on
$\bar\Omega$, in view of  Lemma \ref{CoeffFormula} (iv) we get: $\phi_k\geq (\|
D\|_{\mathcal{C}^0(\bar\Omega_r)}+\|\mathfrak{D}\|_0)/2$.
We now find the bounds on the first norms of the smooth
positive functions $a_k \doteq \sqrt{\phi_k}$. We
begin by noting that $\min_{x\in\bar\Omega} a_k(x) \geq \frac{\|D\|_{\mathcal{C}^0(\bar\Omega_r)}^{1/2}}{\sqrt{2}}$.  
In view of \eqref{wModEst}, \eqref{NormDMol} and since $\nabla
a_k=\frac{\nabla\phi_k}{2a_k}$ and $\nabla^2 a_k =
\frac{\nabla^2\phi_k}{2a_k} - \frac{\nabla a_k\otimes \nabla
  a_k}{a_k}$, Lemma \ref{CoeffFormula} yields:
\begin{equation}\label{farfar}
\begin{split}
\|a_k\|_0 & \leq
\Big(\frac{5\sqrt{3}}{8}\|\tilde{\mathfrak{D}}\|_0\Big)^{1/2} \leq
(4.1)\|\bar D\|_{\mathcal{C}^0(\bar\Omega_r)}^{1/2},\\
\|\nabla a_k\|_0 & \leq\frac{\|\nabla\phi_k\|_0}{2\;\displaystyle{\min_{x\in\bar\Omega}
    a_k(x)}} \leq \frac{5\sqrt{3}}{8\sqrt{2}}\frac{\|\nabla
  \tilde{\mathfrak{D}}\|_0}{\|D\|_{\mathcal{C}^0(\bar\Omega_r)}^{1/2}}
\leq 6 \frac{1}{l}\| D\|_{\mathcal{C}^0(\bar\Omega_r)}^{1/2},\\  
\|\nabla^2 a_k\|_0 & \leq \frac{\|\nabla^2\phi_k\|_0 + 2\|\nabla
     a_k\|_0^2}{2\;\displaystyle{\min_{x\in\bar\Omega} a_k(x)}}\leq \frac{5\sqrt{3}}{8\sqrt{2}}\frac{\|\nabla^2
  \tilde{\mathfrak{D}}\|_0}{\| D\|_{\mathcal{C}^0(\bar\Omega_r)}^{1/2}} +
\sqrt{2}\frac{\|\nabla
  a_k\|_0^2}{\|D\|_{\mathcal{C}^0(\bar\Omega_r)}^{1/2}} \leq (88.8)\frac{1}{l^2}\| D\|_0^{1/2},\\
\|\nabla^3 a_k\|_0 & \leq \frac{\|\nabla^3\phi_k\|_0 + 6\|\nabla  a_k\|_0\|\nabla^2
     a_k\|_0}{2\;\displaystyle{\min_{x\in\bar\Omega} a_k(x)}}\leq 
\frac{5\sqrt{3}}{8\sqrt{2}}\frac{\|\nabla^3
  \tilde{\mathfrak{D}}\|_0}{\| D\|_{\mathcal{C}^0(\bar\Omega_r)}^{1/2}} +
3\sqrt{2}\frac{\|\nabla
  a_k\|_0\|\nabla^2a_k\|_0}{\|D\|_{\mathcal{C}^0(\bar\Omega_r)}^{1/2}}
\\ & \leq (2775.6)\frac{1}{l^3}\|D\|_{\mathcal{C}^0(\bar\Omega_r)}^{1/2}.
\end{split} 
\end{equation}

\smallskip
 
{\bf 3. Iteration of convex integration.}
We set $v_0 \doteq \mathfrak{v}$, $w_0 \doteq \tilde{\mathfrak{w}}$ restricted to
$\bar \Omega$ and then define recursively $v_k\in \mathcal{C}^3(\bar{\Omega})$, $w_k\in
\mathcal{C}^2(\bar{\Omega},\R^2)$ for $k = 1,2,3$ by applying Proposition \ref{OneStepMod}.
We apply it to $v_{k-1}$, $w_{k-1}$, with $a_{k}$ from the decomposition of
$\tilde{\mathfrak{D}}$ in the basis given by $\{\eta_k\}_{k=1}^3$. Lastly,
we set the parameters: 
\begin{equation*}
l_k \doteq \frac{l}{\sigma^{k-1}}<1 , \quad \lambda_k \doteq \frac{1}{l_{k+1}} > \frac{1}{l_k},
\end{equation*}
and the non-decreasing triple $\{\delta_k\}_{k=1}^3$ with the initial choice:
\begin{equation}\label{DeltaChoice}
\delta_1 \doteq \max_{m=1,2}\left\{l^m\|\nabla^{m+1} \mathfrak{v}\|_0\right\}+
\max_{m=0\ldots3, ~k=1\ldots3}\left\{l^m\|\nabla^m a_k\|_0\right\}. 
\end{equation}
The construction is complete by setting $\tilde{v} \doteq v_3$ and
$\tilde{w} \doteq w_3$, and claiming that these satisfy the error
bounds (\ref{DefectStage2})-(\ref{Norm2Est}). 

First, we check that the assumptions of Proposition \ref{OneStepMod} hold at every step.  
The condition that $l_k\in(0,1)$ is easily verified as $l<1$ and $\sigma>1$. 
By (\ref{farfar}) we get $l^m\|\nabla^{m+1} \mathfrak{v}\|_0\leq
(3.1)\| D\|_{\mathcal{C}^0(\bar\Omega_r)}^{1/2}$  for $m=1,2$
and $l^m\|\nabla^m a_k\|_0\leq (2775.6)\| D\|_{\mathcal{C}^0(\bar\Omega_r)}^{1/2}$ for
$m,k=0\ldots 3$, which yield:
$$\delta_1 \leq (2778.7)\| D\|_{\mathcal{C}^0(\bar\Omega_r)}^{1/2}<1$$
by (\ref{DefectSmall}). The first condition in \eqref{DeltaLConditions} is clearly satisfied
for all $k$ by \eqref{DeltaChoice}, given that $l_k<l$ and
$\delta_1\leq \delta_k$. Further, by induction on $k$ and using
\eqref{OneStepMod2}, we obtain:  
\begin{equation*}
\begin{split}
  \|\nabla^{m+1} v_k\|_0 &\leq\|\nabla^{m+1} v_{k-1}\|_0 + \|\nabla^{m+1} v_k-\nabla^{m+1} v_{k-1}\|_0
  \leq\frac{\delta_k}{l^m_k}+ (123)\delta_k\lambda_k^m\\ 
  &\leq\delta_k\frac{124}{l_{k+1}^m}\leq\frac{\delta_{k+1}}{l_{k+1}^m}
  \qquad \mbox{for all } ~ m,k=1,2,
\end{split}
\end{equation*}
if only $\delta_{k+1} = (124)\delta_k$ for $k=1,2$. 
This ensures the second condition in \eqref{DeltaLConditions} in view of
(\ref{DeltaChoice}), provided that $(2778.7)\cdot
(124)^2\delta_0^{1/2}<1$ to get $\delta_1, \delta_2, \delta_3<1$. This
last inequality is implied by the original bound on $\delta_0$.
We omit the verification of estimates \eqref{DefectStage2},
(\ref{Norm0Est}) (note that here we use the condition $0\in\Omega$),
\eqref{Norm1Est} and \eqref{Norm2Est} as they follow directly as in \cite{LP}.
\end{proof}
 
\section{A proof of Theorem \ref{weakMA}}\label{sec5}

We start by showing the main approximation result needed
in Theorem \ref{weakMA}, that is a version of the result in Theorem
\ref{w2oi-hld}. It consists of iterating the ``stages'' construction,
with the sole restrictive assumption on the  smallness of the initial deficit.
 
\begin{proposition}\label{prop2.4.4}
Let $\Omega\subset\Omega_r\subset\mathbb{R}^2$ be open bounded sets,
where $\Omega_r = \Omega+ B_r(0)$ for $r>0$. Let: 
$$\delta_0 < \min\Big\{\frac{r}{2}, (5.4) 10^{-16}\Big\}.$$ 
Given three functions: $v\in \mathcal{C}^2(\bar\Omega_r)$, $w\in\mathcal{C}^2(\bar
\Omega_r,\mathbb{R}^2)$ and $A\in\mathcal{C}^{0, \beta} (\bar\Omega_r,
\mathbb{R}^{2\times 2}_{sym})$ of H\"older regularity $\beta\in
(0,1)$, assume that:
$$D\doteq A - \big(\frac{1}{2}\nabla v\otimes \nabla v +
\mathrm{sym}\nabla w\big), \qquad 0< \|D\|_{\mathcal{C}^0(\bar\Omega_r)}\leq\delta_0.$$
Then, for every exponent
$ \alpha\in (0,\min\big\{ \frac{1}{7},\frac{\beta}{2} \big\})$ one can find 
$\bar{v}\in \mathcal{C}^{1,\alpha}(\bar{\Omega})$, 
$\bar{w}\in \mathcal{C}^{1,\alpha}(\bar{\Omega},\R^2)$ such that:
$\frac{1}{2}\nabla\bar{v}\otimes\nabla\bar{v} + \mathrm{sym}\nabla \bar{w}=A$, and:
\begin{equation*}
\|v-\bar{v}\|_0<(0.21)\delta_0, \qquad \|w-\bar{w}\|_0<\big(0.5 +
(63.4)\cdot \mathrm{diam}\,\Omega_r\big)(1+\|\nabla v_0\|_{\mathcal{C}^0(\bar\Omega_r)})\delta_0.
\end{equation*}
\end{proposition}

\begin{proof}
{\bf 1.} Define a decreasing sequence of open domains
$\{\Omega^k\}_{k=0}^\infty$ by setting:
$$\Omega^0\doteq \Omega_r,\qquad \Omega^k \doteq
\Omega_{r-\delta_0\sum_{i=1}^k 2^{-i}}~~\mbox{ for } k\geq 1,$$
so that: $\Omega^k = \Omega^{k+1}+B_{2^{-(k+1)}\delta_0}(0)$. Let
$v_0\doteq v$ and $w_0\doteq w$ on $\bar\Omega^0$. Given $k\geq 0$ and
$v_k\in\mathcal{C}^2(\bar\Omega^k)$,
$w_k\in\mathcal{C}^2(\bar\Omega^k,\mathbb{R}^2)$, resulting in the
nonzero deficit:
$$D_k\doteq A - \big( \frac{1}{2}\nabla v_k\otimes \nabla v_k +
\mathrm{sym}\,\nabla w_k),$$
we will construct the functions $v_{k+1}\in\mathcal{C}^2(\bar\Omega^{k+1})$,
$w_{k+1}\in\mathcal{C}^2(\bar\Omega^{k+1},\mathbb{R}^2)$ 
by applying Proposition \ref{Stage3} to the sets $\Omega^{k+1}\subset\Omega^k$, the matrix
field $A_{\mid \bar\Omega^k}$ and the parameters $\sigma_k, M_k$
chosen according to the procedure indicated below.
First, choose the exponent $s\in (0,1)$ to satisfy:
\begin{equation}\label{s}
\frac{6\alpha}{1-\alpha}<s<\frac{6\beta}{2-\beta}, 
\end{equation}
Existence of such $s$ in guaranteed by $\alpha\in (0,\min\big\{
\frac{1}{7},\frac{\beta}{2} \big\})$. Second, set the constant: 
\begin{equation}\label{C}
\mathfrak{C}\doteq (20.9) 10^8 \big(1+\|\nabla
v_0\|_{\mathcal{C}(\bar\Omega^0)}\big)
\end{equation}
and let $\{\sigma_k\}_{k=0}^\infty$ be an increasing sequence of positive
numbers, converging to some $\sigma_{max}$ and satisfying:
\begin{equation}\label{sigma}
\sigma_k^s\geq \frac{16}{9}, \quad \sigma_k^{1-s}>(3.7) 10^{15} \quad \mbox{
  for all } k\geq 0 \qquad \mbox{and} \qquad (\sigma_{max})^{\frac{s}{2}(1-\alpha)-3\alpha}>\mathfrak{C}^\alpha.
\end{equation}
We note that it is enough to take $\sigma_k=\sigma_{max}$ for all $k$,
but having $\sigma_k$ as small as possible is advantageous for the
numerical calculations. Third, the constants $\{M_k\}_{k=0}^\infty$ are defined by:
\begin{equation}\label{M}
\begin{split}
& M_{k}\doteq M_0 \mathfrak{C}^k \prod_{j=0}^{k-1}\sigma_j^3 \qquad
\mbox{ for all } k\geq 1 \\ & \mbox{and} \qquad M_0\doteq\left\{\begin{array}{ll}
N 2^{\frac{1}{\beta}} (\sigma_{max})^{\frac{1}{\beta}} \|A\|_{\mathcal{C}^{0,\beta}(\bar\Omega^0)}^{\frac{1}{\beta}}
\|D_0\|_{\mathcal{C}^{0,\beta}(\bar\Omega^0)}^{\frac{1}{2} -
  \frac{1}{\beta}} &\mbox{if } A\neq 0\\
N &\mbox{if } A = 0,
\end{array}\right. 
\end{split}
\end{equation}
for a large constant $N\geq 1$ below (that is evaluated numerically in the
following implementation). In par\-ti\-cu\-lar, there holds:
\begin{equation}\label{MM}
M_0\geq {2}{\delta_0^{-\frac{1}{2}}}.
\end{equation}

\smallskip

{\bf 2.} We now inductively prove that:
\begin{equation}\label{D}
\|D_{k}\|_{\mathcal{C}^0(\bar\Omega^{k})}\leq
\frac{\|D_0\|_{\mathcal{C}^0(\bar\Omega^{0)}}}{\prod_{j=0}^{k-1}\sigma_j^s}
  \qquad \mbox{for all } k\geq 1.
\end{equation}
For $k=1$, Proposition \ref{Stage3} gives:
$$\frac{\|D_{1}\|_{\mathcal{C}^0(\bar\Omega^{1})}}{
\|D_0\|_{\mathcal{C}^0(\bar\Omega^{0})}} \leq \frac{
\|A\|_{\mathcal{C}^{0, \beta}(\bar\Omega^{0})} \|D_0\|_{\mathcal{C}^0(\bar\Omega^0)}^{\frac{\beta}{2}-1}}{M_0^\beta}
+ \frac{(1.9)10^{15}}{\sigma_0}.$$
The second term in the right hand side above is majored by
$\frac{1}{2\sigma_0^s}$ by (\ref{sigma}), whereas the first term is also
bounded by: $\frac{1}{2 N^\beta\sigma_0^s}\leq \frac{1}{2\sigma_0^s}$
in view of (\ref{M}). For $k\geq 1$, we similarly observe that:
$$\frac{\|D_{k+1}\|_{\mathcal{C}^0(\bar\Omega^{k+1})}}{
\|D_k\|_{\mathcal{C}^0(\bar\Omega^{k})}} \leq \frac{
\|A\|_{\mathcal{C}^{0, \beta}(\bar\Omega^{0})} \|D_k\|_{\mathcal{C}^0(\bar\Omega^{k}}^{\frac{\beta}{2}-1}}{M_k^\beta}
+ \frac{(1.9)10^{15}}{\sigma_k}\leq \frac{1}{\sigma_k^s},$$
where the bound on the second term follows as in the case $k=0$, while
to estimate the first term:
\begin{equation*}
\frac{\|A\|_{\mathcal{C}^{0, \beta}(\bar\Omega^{0})}
  \|D_k\|_{\mathcal{C}^0(\bar\Omega^{k})}^{\frac{\beta}{2}-1}}{M_k^\beta}
\leq \frac{1}{2N
  \sigma_{max}^s\mathfrak{C}^{k\beta}\sigma_k^{3\beta}\prod_{j=0}^{k-1}\sigma_j^{3\beta
  + s(\frac{\beta}{2}-1)}}\leq \frac{1}{2\sigma_k^s},
\end{equation*}
we used the inductive assumption, (\ref{M}) and (\ref{s}), implying that $3\beta + s(\frac{\beta}{2}-1)>0$.
This ends the proof of (\ref{D}).

Observe also that by (\ref{D}), (\ref{sigma}) and the assumed bound on
$\delta_0$ we get, for all $k\geq 0$:
\begin{equation}\label{vk}
\begin{split}
1+ \|\nabla v_k\|_{\mathcal{C}^0(\bar\Omega^k)}& \leq 1+\|\nabla
v_0\|_{\mathcal{C}^0(\bar\Omega^0)} + (1.1)
10^{8}\sum_{j=0}^\infty\|D_k\|_{\mathcal{C}^0(\bar\Omega^k)}^{\frac{1}{2}}
\\ & \leq 
1+\|\nabla v_0\|_{\mathcal{C}^0(\bar\Omega^0)} + (1.1)
10^{8} \|D_0\|_{\mathcal{C}^0(\bar\Omega^0)}^{\frac{1}{2}}\Big(\sum_{j=0}^\infty 
\frac{1}{\prod_{i=0}^{j-1}\sigma_i^{s}}\Big)^{\frac{1}{2}} \\ & 
\leq  1+\|\nabla v_0\|_{\mathcal{C}^0(\bar\Omega^0)} + (1.1)
10^{8} \cdot \delta_0^{\frac{1}{2}}\sum_{j=0}^\infty (\frac{3}{4})^j
\\ & \leq 1+\|\nabla v_0\|_{\mathcal{C}^0(\bar\Omega^0)} + 1.2 
\leq (2.2) \big(1+\|\nabla v_0\|_{\mathcal{C}^0(\bar\Omega^0)}\big).
\end{split}
\end{equation}

\smallskip

{\bf 3.} Clearly, when $D_k=0$, we stop the recursive process
and set $(\bar v, \bar w)\doteq (v_k, w_k)$ with the claimed error estimates
established in the next step below.  We now proceed by validating the
assumptions of Proposition \ref{Stage3}  in case $D_k\neq 0$. 

The bound $\|D_k\|_{\mathcal{C}^0(\bar\Omega^k)}\leq \delta_0$ is a
consequence  of (\ref{D}) and $\sigma_j\geq 1$ for all $j\geq 0$. Similarly:
$$\frac{2^{k+1}}{\delta_0}\|D_k\|_{\mathcal{C}^0(\bar\Omega^k)}\leq M_k$$
follows by (\ref{MM}) for $k=0$, whereas for $k\geq 1$ we additionally use (\ref{D}) in:
$$\frac{2^{k+1}}{\delta_0}\|D_k\|_{\mathcal{C}^0(\bar\Omega^k)}\leq
\frac{2^{k+1}}{\delta_0}\|D_0\|_{\mathcal{C}^0(\bar\Omega^0)}\leq
\frac{2^{k+1}}{\delta_0^{\frac{1}{2}}}\leq 2^k M_0\leq M_k.$$
The estimates: $\|\nabla^2 v_0\|_{\mathcal{C}^0(\bar\Omega^0)}, 
\|\nabla^2 w_0\|_{\mathcal{C}^0(\bar\Omega^0)}\leq M_0$ are valid if
$N$ is chosen large enough. Finally, for all $k\geq 1$ we have:
\begin{equation*}
\begin{split}
& \|\nabla^2 v_k\|_{\mathcal{C}^0(\bar\Omega^k)}\leq (7.3) 10^8 M_{k-1}\sigma_{k-1}^3 \leq M_k,\\
& \|\nabla^2 w_k\|_{\mathcal{C}^0(\bar\Omega^k)} \, \leq (9.5) 10^8
M_{k-1}\sigma_{k-1}^3 \big(1+\|\nabla
v_{k-1}\|_{\mathcal{C}(\bar\Omega^{k-1})}\big) \\
&\qquad \qquad \quad \quad\leq (20.9) 10^8  M_{k-1}\sigma_{k-1}^3 \big(1+\|\nabla
v_{0}\|_{\mathcal{C}(\bar\Omega^{0})}\big)\leq M_k,
\end{split}
\end{equation*}
in virtue of (\ref{M}), (\ref{vk}), (\ref{C}) and  Proposition \ref{Stage3}.

\smallskip

{\bf 4.} We now show that the sequences $\{v_k\}_{k=0}^\infty$,
$\{w_k\}_{k=0}^\infty$  converge in $\mathcal{C}^{1,
  \alpha}(\bar\Omega)$ to some $\bar v, \bar w$. By (\ref{D}) this will
imply that $A- \big(\frac{1}{2}\nabla \bar v\otimes \nabla \bar v +
\mbox{sym}\nabla \bar w\big)=\lim_{k\to\infty} D_k = 0$ together with 
the desired estimates on $\|\bar v - v\|_0$ and $\|\bar w - w\|_0$. Indeed:
\begin{equation*}
\begin{split}
\|v_k - v_0\|_0 &\leq \sum_{i=1}^{k-1}\|v_{i+1}-v_i\|_0 \leq
\sum_{i=0}^\infty \frac{(1.8)10^7}{M_i} \|D_i\|_{\mathcal{C}^0(\bar\Omega^i)}\\ & \leq 
(1.8) 10^7 \bigg(\frac{\|D_0\|_{\mathcal{C}^0(\bar\Omega^0)}}{M_0} +
\sum_{i=1}^\infty \frac{\|D_0\|_{\mathcal{C}^0(\bar\Omega^0)}}{M_i\prod_{j=0}^{i-1}\sigma_j^s}\bigg) 
\leq  (1.8) 10^7 \frac{\|D_0\|_{\mathcal{C}^0(\bar\Omega^0)}}{M_0}
\bigg( 1 + \sum_{i=1}^\infty \frac{1}{\mathfrak{C}^i}\bigg) \\ & = (1.8) 10^7 \frac{\delta_0}{M_0}
\frac{\mathfrak{C}}{\mathfrak{C}-1} \leq (1.8) 10^7
\frac{\delta_0^{\frac{1}{2}}}{2}\cdot \delta_0\leq (0.21)\delta_0, 
\end{split}
\end{equation*}
by Proposition \ref{Stage3}, (\ref{D}), (\ref{M}). Automatically, the
bound above implies that $\{v_k\}_{k=0}^\infty$ is Cauchy in
$\mathcal{C}^0(\bar\Omega)$. The similar statements for
$\{w_k\}_{k=0}^\infty$ follow by (\ref{vk}) in:
\begin{equation*}
\begin{split}
\|w_k - w_0\|_0 &\leq \sum_{i=1}^{k-1}\|w_{i+1}-w_i\|_0 \leq
\sum_{i=0}^\infty \Big(\frac{(1.8)10^7}{M_i} +
(12.6) \cdot \mbox{diam}\,\Omega^0\Big) \|D_i\|_{\mathcal{C}^0(\bar\Omega^i)}
\big(1+\|\nabla v_i\|_{\mathcal{C}^0(\bar\Omega^i)}\big)
\\ & \leq  (2.2) \big(1+\|\nabla
v_0\|_{\mathcal{C}^0(\bar\Omega^0)}\big)\Big((12.6)(\mbox{diam}\,\Omega^0)
\sum_{i=0}^\infty \|D_i\|_{\mathcal{C}^0(\bar\Omega^i)} +
(0.21)\delta_0\Big) \\ & \leq  \delta_0 (2.2) \big(1+\|\nabla
v_0\|_{\mathcal{C}^0(\bar\Omega^0)}\big)\Big((12.6) \frac{16}{7}
(\mbox{diam}\,\Omega^0) + 0.21\Big) \\ & \leq \delta_0 \big(1+\|\nabla
v_0\|_{\mathcal{C}^0(\bar\Omega^0)}\big)\Big( (63.4)
(\mbox{diam}\,\Omega^0) + 0.5\Big).
\end{split}
\end{equation*}
The fact that $\{v_k\}_{k=0}^\infty$,  $\{w_k\}_{k=0}^\infty$ are Cauchy in
$\mathcal{C}^1(\bar\Omega)$ is obtained in:
\begin{equation*}
\begin{split}
& \|\nabla v_{k+1}-\nabla v_k\|_{0}\leq (1.1) 10^8 \Big(\frac{\delta_0}{\prod_{j=0}^{k-1}\sigma_j^s}\Big)^{\frac{1}{2}}
\leq \big(\frac{3}{4}\big)^{k}(1.1) 10^8 \delta_0^{\frac{1}{2}}\\
& \|\nabla w_{k+1}-\nabla w_k\|_{0}\leq \big(\frac{3}{4}\big)^{k}(1.1)
10^8 \delta_0^{\frac{1}{2}} \cdot (2.2) \big(1+\|\nabla v_0\|_{\mathcal{C}^0(\bar\Omega^0)}\big).
\end{split}
\end{equation*}
Finally, we estimate the $\mathcal{C}^{1,\alpha}$ norms by
interpolating from $\mathcal{C}^{1}$ to $\mathcal{C}^{2}$ norms in:
\begin{equation*}
\begin{split}
[\nabla v_{k+1}-\nabla v_k]_{\alpha} & \leq C \|\nabla v_{k+1}-\nabla v_k\|_{\mathcal{C}^0(\bar\Omega^{k+1})}^{1-\alpha}
\|\nabla^2 v_{k+1}-\nabla^2 v_k\|_{\mathcal{C}^0(\bar\Omega^{k+1})}^{\alpha} \\ & \leq 
C \frac{1}{\prod_{j=0}^{k-1} \sigma_j^{\frac{s}{2}(1-\alpha)}} \cdot
\sigma_k^{3\alpha}\cdot \Big(\mathfrak{C}^k \prod_{j=0}^{k-1}\sigma_j^3\Big)^\alpha
\leq C \sigma_{max}^{3\alpha}\cdot \mathfrak{C}^k
\prod_{j=0}^{k-1}\sigma_j^{3\alpha - \frac{s}{2}(1-\alpha)} \\ & \leq
C \big(\mathfrak{C}\cdot (\sigma_{max})^{3\alpha - \frac{s}{2}(1-\alpha)}\big)^k,\\
[\nabla w_{k+1}-\nabla w_k]_{\alpha}  & \leq C  \|\nabla w_{k+1}-\nabla w_k\|_{\mathcal{C}^0(\bar\Omega^{k+1})}^{1-\alpha}
\|\nabla^2 w_{k+1}-\nabla^2
w_k\|_{\mathcal{C}^0(\bar\Omega^{k+1})}^{\alpha} \\ & \leq 
C (1+\|\nabla v_{0}\|_{\mathcal{C}^0(\bar\Omega^{0})})^{1-\alpha}
\big(\mathfrak{C}\cdot (\sigma_{max})^{3\alpha - \frac{s}{2}(1-\alpha)}\big)^k.
\end{split}
\end{equation*}
Above, the constant $C$ depends, in its first appearance, only on the
curvature of a smooth superset of $\bar\Omega$ contained in
$\Omega_r$. Eventually, $C$ depends on various 
numerical values including $M_0$, $\alpha$ and $\sigma_{max}$ but it
is independent of $k$. We see that the last condition in (\ref{sigma})
implies that the constant of the geometric progression in the right
hand sides is less than $1$ for large $k$, that is when $\sigma_k$ is
close to $\sigma_{max}$. This establishes that $\{v_k\}_{k=0}^\infty$,
$\{w_k\}_{k=0}^\infty$  are Cauchy sequences in $\mathcal{C}^{1,
  \alpha}(\bar\Omega)$ and completes the proof of Proposition \ref{prop2.4.4}.
\end{proof} 

\bigskip

\noindent {\bf Proof of Theorem \ref{w2oi-hld}. }
 
\smallskip

Let $f, v_0$ be as in Theorem \ref{weakMA}. Fix a small parameter
$\epsilon>0$. We first extend
$v_0$ to a continuous function on $\mathbb{R}^2$ and approximate this extension
with $\bar v_0\in\mathcal{C}^\infty(\mathbb{R}^2)$ satisfying:
$$\|\bar v_0 - v_0\|_0\leq \epsilon.$$ 
Let $A=(\lambda+c)\mbox{Id}_2\in\mathcal{C}^{0,\beta}(\bar\Omega, \mathbb{R}^{2\times
  2}_{sym})$ where $c>0$ is a constant and $\lambda\in
\mathcal{C}^{0,\beta}(\bar V)$ solves: $-\Delta\lambda =
\chi_{\bar\Omega} f$ in an open superset $V$ of $\bar\Omega$, with
$\lambda = 0$ on $\partial V$. Then, as discussed in the introduction, $A=\frac{1}{2}\nabla v_0\otimes
\nabla v_0 + \mbox{sym} \nabla w_0$ holds in $\bar\Omega_r$ for some
$w_0\in\mathcal{C}^1(\bar\Omega_r, \mathbb{R}^2)$ on $\Omega_r=\Omega
+ B_r(0)$ for some $r>0$. Note that if $c$ is large enough to guarantee:
\begin{equation*}
\begin{split}
& \frac{3}{4}(\lambda +c) -\frac{1}{2}(\partial_1\bar v_0)^2 +
\frac{1}{8}(\partial_2 \bar v_0)^2 - \partial_1w_0^1 + \frac{1}{4}\partial_2w_0^2>0,\\
& (\lambda+c) -\frac{1}{2}(\partial_2\bar v_0)^2 -
(\partial_1 \bar v_0)(\partial_2 \bar v_0) - \partial_2w_0^2
- \partial_1w_0^2 - \partial_2w_0^1>0, \qquad \mbox{in }~ \bar\Omega_r,\\
& (\lambda+c) -\frac{1}{2}(\partial_2\bar v_0)^2 +
(\partial_1 \bar v_0)(\partial_2 \bar v_0) - \partial_2w_0^2 + \partial_1w_0^2 + \partial_2w_0^1>0,
\end{split}
\end{equation*}
then in view of Lemma \ref{CoeffFormula} it follows that:
$$ A - \big(\frac{1}{2}\nabla \bar v_0\otimes \nabla \bar v_0 + \mbox{sym}
\nabla w_0\big) = \sum_{k=1}^3\phi_k \eta_k\otimes \eta_k, \qquad
\phi_k>0 ~~\mbox{ in } \bar\Omega_r.$$
Clearly, the field $w_0$ may be exchanged with a smooth field $\bar
w_0\in\mathcal{C}^\infty(\bar\Omega_r, \mathbb{R}^2)$ so that the
positive coefficient decomposition above is still valid.
Applying now Theorem \ref{Approx2}, we get
$v\in\mathcal{C}^\infty(\bar\Omega_r)$ and $w\in\mathcal{C}^\infty(\bar\Omega_r, \mathbb{R}^2)$ 
such that:
$$\|v - \bar v_0\|_0\leq \epsilon \quad \mbox{ and }\quad
\|A-\big(\frac{1}{2}\nabla v\otimes \nabla v + \mbox{sym} \nabla
w\big) \|_0\leq\min\big\{\epsilon, r, (5.4) 10^{-16}\big\}.$$
In virtue of Proposition \ref{prop2.4.4}, we further find
$\bar v\in\mathcal{C}^{1,\alpha}(\bar\Omega)$ and
$\bar w\in\mathcal{C}^{1,\alpha}(\bar\Omega, \mathbb{R}^2)$  solving:
$\frac{1}{2}\nabla \bar v\otimes \nabla \bar v + \mbox{sym} \nabla 
\bar w  = A$ in $\bar\Omega$, with the approximation error:
$$\|\bar v - v\|_0\leq (0.21)\epsilon.$$
Concluding, the function $\bar v$ is a H\"older regular solution to (\ref{MA}), approximating the
given $v_0$ with the arbitrarily small error:
$$\|\bar v - v_0\|_0\leq \|\bar v - v\|_0+\|v - \bar v_0\|_0+\|\bar v_0 -
v_0\|_0\leq 3\epsilon,$$
as claimed in the result.
\endproof
 
 \section{Numerical implementation of the $\mathcal{C}^{1,\alpha}$
   convergence scheme}\label{num2_sec}

In section \ref{Numerical1} we implemented the approximation
construction of Theorem \ref{Approx2} and Proposition \ref{Stage2}; 
we now implement the H\"older approximation
specified in Proposition \ref{Stage3} and Proposition \ref{prop2.4.4}.
Visualizations relevant to the $\mathcal{C}^{1,\alpha}$ case are much harder to obtain. 
The main difficulties come from the smallness of various parameter
values and and the fact that the ratio $\sigma$ between the
frequencies of subsequent corrugations is by necessity very large. 
This first obstacle was overcome through the use of the Python package
mpmath \cite{mpmath}, allowing the user to define floating point arithmetics up to an arbitrary precision. 
The second obstacle of the necessity of large $\sigma$ is insurmountable in most cases. 
Nonetheless, in the example of the 
degenerate Monge-Amp\'{e}re equation $\det\nabla^2 v = 0$, 
small values of $\lambda$ were found to produce a reduction in
the deficit which justified the visualization obtained below.  

The approach taken was significantly different from the case of $\mathcal{C}^1$ convergence.  
In each of the examples the problem was solved explicitly through symbolic calculations,
defining symbolic variables and parameters in MATLAB (the only
calculation which could not be done symbolically was the  mollification step).
The resulting text files, containing the outcomes were then modified through a script to make the syntax
compatible with the mpmath package. 
We subsequently sampled these functions at 1000 randomly selected
points on the square domain $\Omega=(-1,1)\times(-1,1)$,  
 and the maximum value attained by each function was recorded. 
The main purpose of these calculations was to study how the results
of convex integration varied by changing the value of $\sigma$. 
We argue that the maximum sample value, while not being the exact supremum of these functions on the domain, 
is a good approximation of the order of magnitude. 
In fact, when the process was repeated over different sets of 1000
sample points, the variation in the values obtained was negligible
when considering the order of magnitude, as illustrated in the table below. 
 
\bigskip

\begin{center}
\begin{tabular}{ |c|c|c|c| } 
    \hline
     Test & $\|\tilde D\|_0$ & $\|v_3\|_0$ & $\|\nabla w_3\|_0$ \\ 
    \hline
    1 & $0.64 \cdot 10^{-11}$ & $0.114 \cdot 10^{-23}$ & $0.53 \cdot 10^{-8}$ \\ 
    \hline
    2 & $0.57 \cdot 10^{-11}$ & $0.118 \cdot 10^{-23}$ & $0.58 \cdot 10^{-8}$ \\ 
    \hline
    3 & $0.62 \cdot 10^{-11}$ & $0.119 \cdot 10^{-23}$ & $0.54 \cdot 10^{-8}$ \\ 
    \hline
    4 & $0.61 \cdot 10^{-11}$ & $0.115 \cdot 10^{-23}$ & $0.49 \cdot 10^{-8}$ \\ 
    \hline
    5 & $0.66 \cdot 10^{-11}$ & $0.119 \cdot 10^{-23}$ & $0.54 \cdot 10^{-8}$ \\ 
    \hline
    6 & $0.70 \cdot 10^{-11}$ & $0.116 \cdot 10^{-23}$ & $0.56 \cdot 10^{-8}$ \\ 
    \hline
    7 & $0.65 \cdot 10^{-11}$ & $0.104 \cdot 10^{-23}$ & $0.62 \cdot 10^{-8}$ \\ 
    \hline
    8 & $0.61 \cdot 10^{-11}$ & $0.118 \cdot 10^{-23}$ & $0.67 \cdot 10^{-8}$ \\ 
    \hline
    9 & $0.61 \cdot 10^{-11}$ & $0.111 \cdot 10^{-23}$ & $0.50 \cdot 10^{-8}$ \\ 
    \hline
    10 & $0.58 \cdot 10^{-11}$ & $0.108 \cdot 10^{-23}$ & $0.76 \cdot 10^{-8}$ \\ 
    \hline
\end{tabular}
\end{center}
 
\vspace{3mm}

In this table we show the results of running our code 10 times on the same example with the same parameters, 
where the only change was the choice of the sample points. 
In each test we evaluated the maximum defect recorded, the maximum
value of $|v|$ and the maximum gradient of $w$.  
As can be seen, the order of magnitude of all these functions remains the same. 
Similar results were obtained with different examples, functions and
parameters, but the details are omitted as they do not add particular
insight.  

\begin{example}\label{Example_11}
We approximate $v_0=0$ and $w_0=0$ with the solution to:
$$\mathcal{D}et\,\nabla^2v = -10^{-18} < 0.$$ 
This results in $A = -10^{-18}(x^2+y^2)\mbox{Id}_2$, coinciding with
the original defect  $\|D\|_0 \sim 10^{-18}\sqrt{2}$. The defect
evaluated for $\sigma = 35$ ($1$st stage of convex integration, third step and third
corrugation added) becomes:  $\|D_3\|_0 \sim 9.7\cdot 10^{-19}$.
\end{example}


\begin{example}\label{Example_12}
We approximate $v_0=0$ and $w_0=0$ with the solution to:
$$\mathcal{D}et\,\nabla^2v = 10^{-18} > 0.$$ 
This results in $A = 10^{-18}(x^2+y^2)\mbox{Id}_2$, coinciding with
the original defect, where $\|D\|_0 \sim 10^{-18}\sqrt{2}$. 
The defect evaluated for $\sigma = 35$ becomes:  $\|D_3\|_0 \sim 9.1\cdot 10^{-19}$.
In Figures \ref{FigC1alpha2} and \ref{Fig2detail} we show visualizations of $v$ on the
subdomains $[1-10^{-19},1]^2$ and $[1-2\cdot 10^{-21},1]^2$, respectively.   
\end{example}

\begin{figure}[h]
\fbox{\includegraphics[height=10cm,trim={2.5cm 1cm 0cm 1cm},clip]{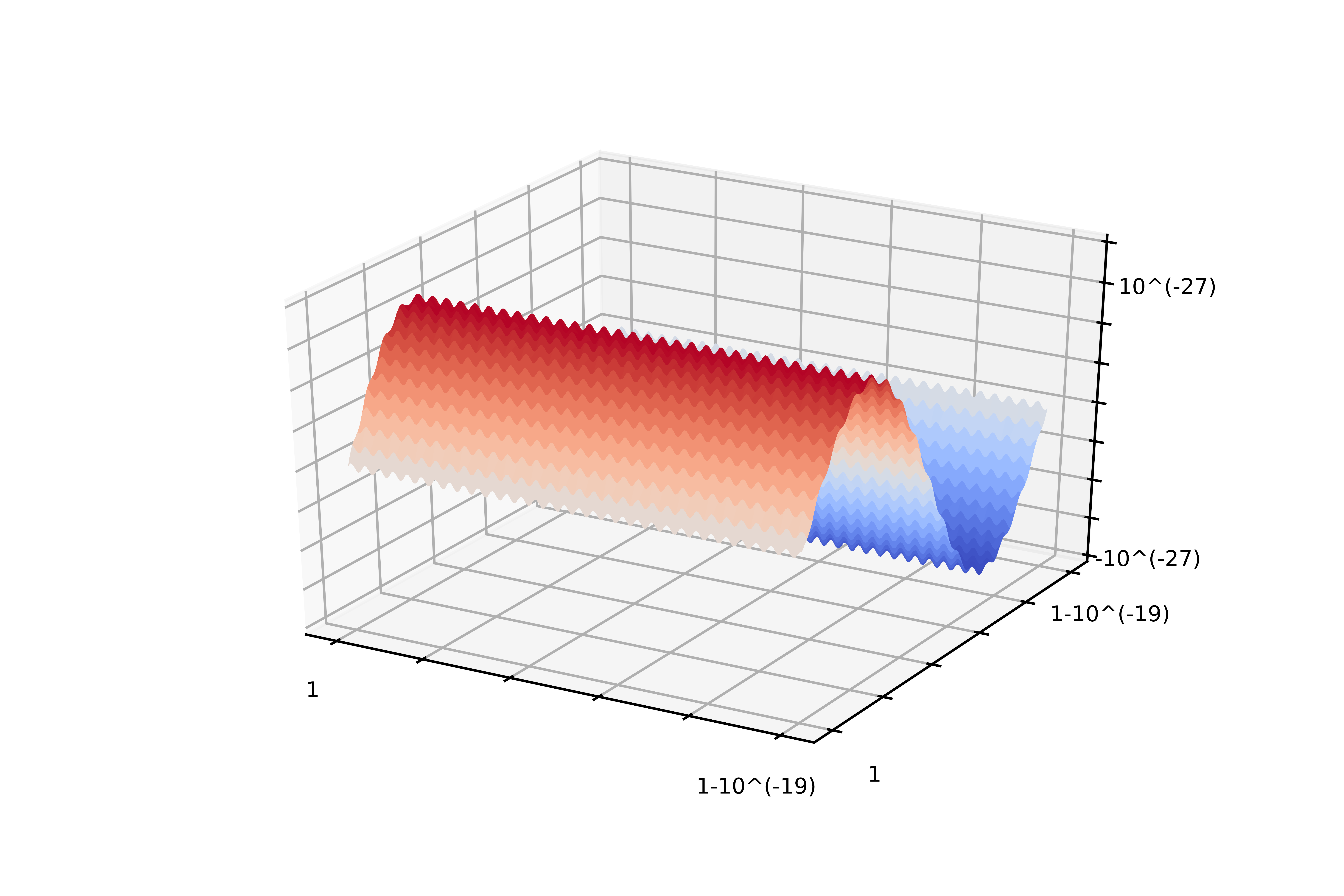}}
\caption{The three corrugations in Example \ref{Example_12}. }
\label{FigC1alpha2}
\end{figure}

\begin{figure}[h]
\fbox{\includegraphics[height=10cm,trim={2.5cm 1cm 0cm 1cm},clip]{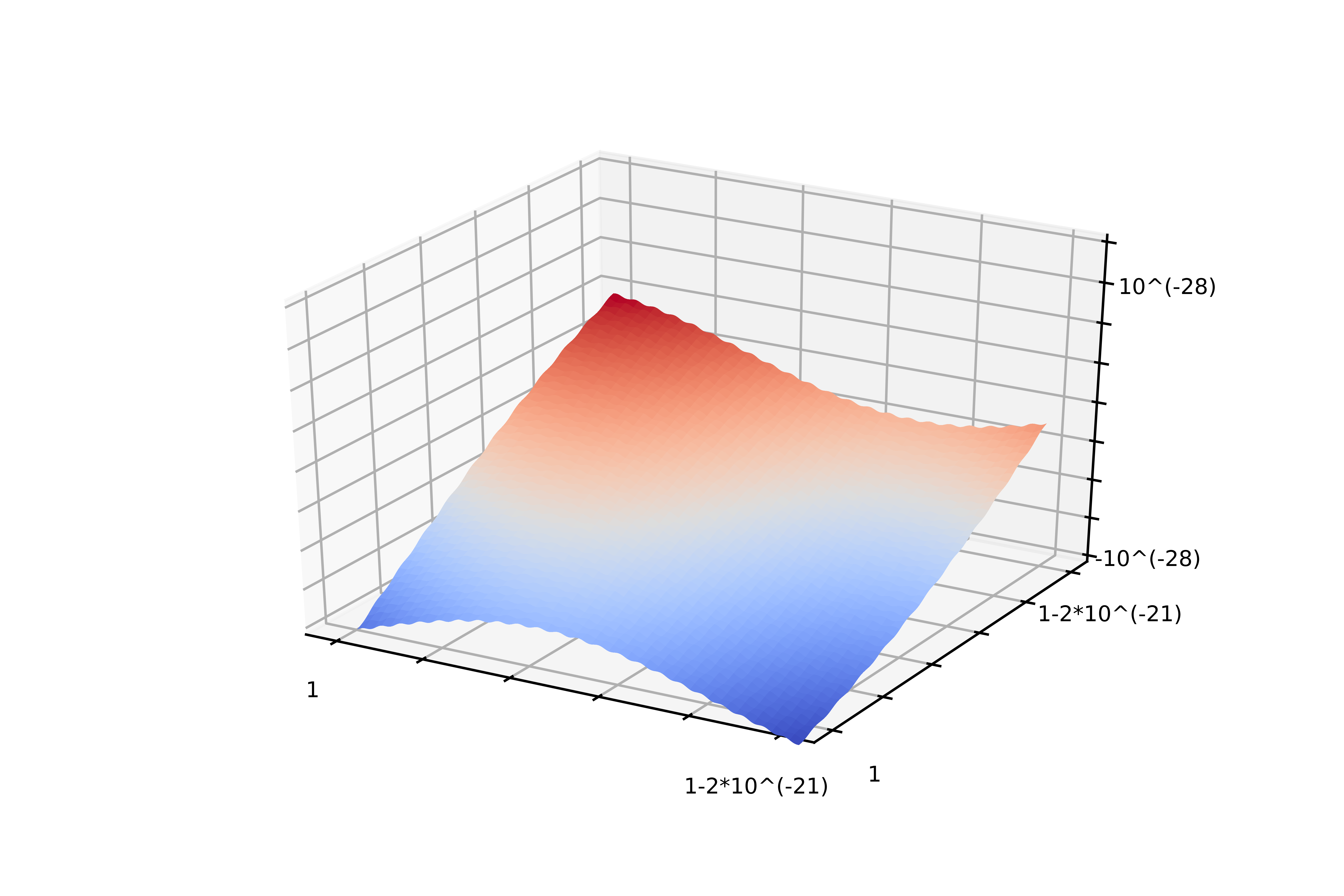}}
\caption{The detail of Figure \ref{FigC1alpha2}. }
\label{Fig2detail}
\end{figure}

\begin{example}\label{Example_13}
We approximate $v_0(x,y) = 10^{-9} (x^2+y^2)$ and $w_0=0$ with the solution to:
$$\mathcal{D}et\,\nabla^2v = 0.$$ 
Thus $A = 0$ and the initial defect is given by:
$ D(x,y) =\frac{1}{2}\nabla v_0\otimes \nabla v_0 =  10^{-18} \cdot 2\cdot (x,y)\otimes (x,y)$ with $ \|D\|_0 \sim
4\cdot 10^{-18}$. The defect evaluated for $\sigma = 35$ becomes:  $\|D_3\|_0 \sim 9.5\cdot 10^{-19}$.
In Figure \ref{FigC1alpha2} we show a visualization of $v$ on the
subdomain $[1-10^{-19},1]^2$. 
\end{example}
  
\begin{figure}[h]
\fbox{\includegraphics[height=10cm,trim={2.5cm 1cm 0cm 1cm},clip]{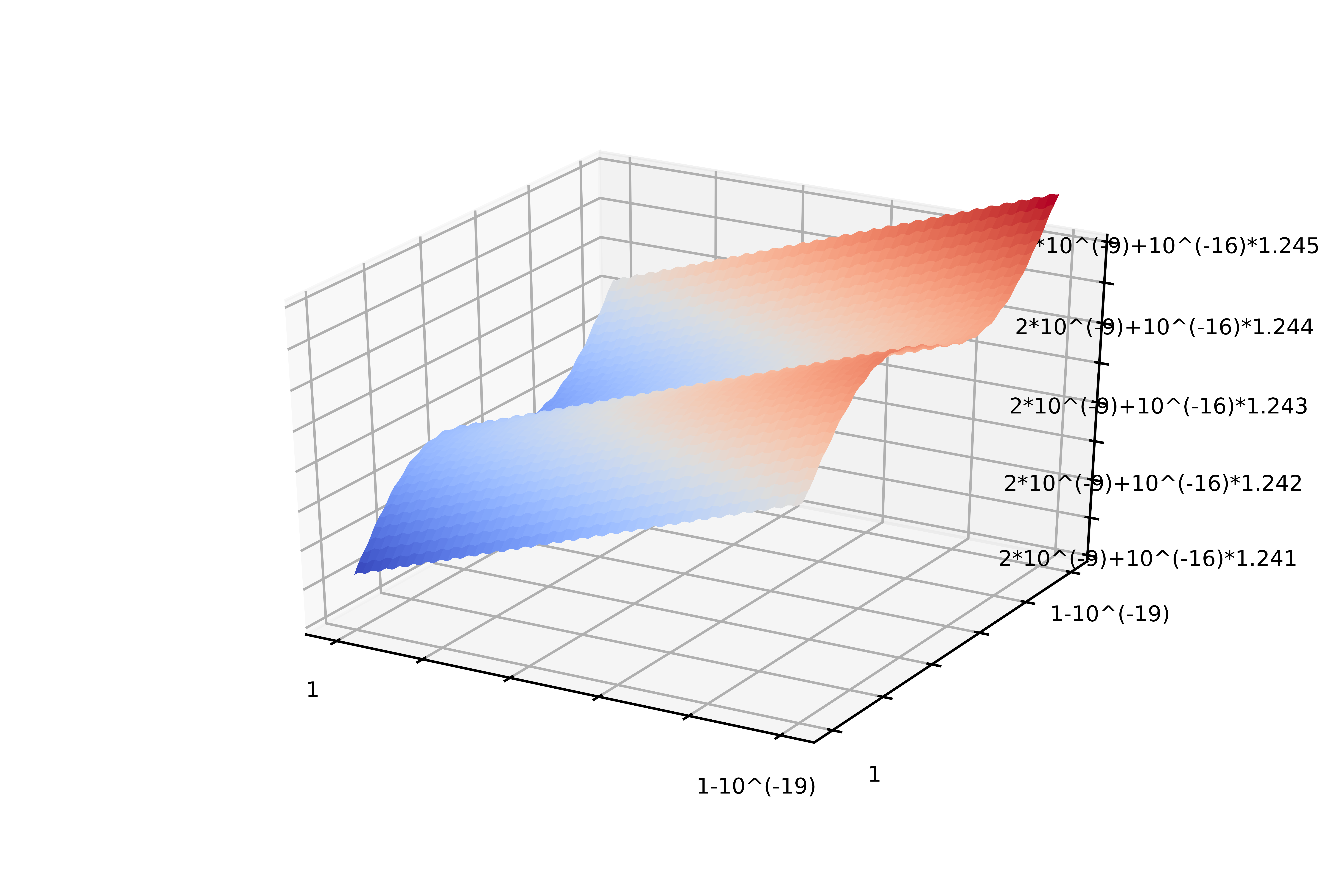}}
\caption{The corrugations in Example \ref{Example_13}.}
\label{FigC1alpha3}
\end{figure}

In each of the examples a value of $\lambda_1 = 10^{19}$ was chosen
and $\sigma$ was increased exponentially from $10^1$ to $10^{18}$,
covering the orders of magnitude between the theoretical minimum of
$\sigma$ and $\sigma_{max}$.  
We then applied the modification of $w$ and the three steps of convex
integration in Proposition \ref{Stage3} and evaluated the maximum of the new defect, the norm of
$v$ together with is gradient and Hessian.
The numerical results obtained are summarized in the tables in Appendix \ref{secapen}.


\appendix
\section{Error tables for numerical calculations of the
  $\mathcal{C}^{1,\alpha}$ implementation}\label{secapen}

   \begin{table}[h]
   \label{normDtable}
   \caption{Values of the defect $\|D_3\|_0$}
  \begin{center}
  \begin{tabular}{ |c|c|c|c| }
    \hline
     $\sigma$ & Example \ref{Example_11} & Example \ref{Example_12} & Example \ref{Example_13}\\ 
    \hline
    $10^1$ & $0.332 \cdot 10^{-17}$ & $0.393 \cdot 10^{-17}$ & $0.327 \cdot 10^{-17}$ \\ 
    \hline
    $10^2$ & $0.316 \cdot 10^{-18}$ & $0.364 \cdot 10^{-18}$ & $0.319 \cdot 10^{-18}$ \\ 
    \hline
    $10^3$ & $0.318 \cdot 10^{-19}$ & $0.376 \cdot 10^{-19}$ & $0.332 \cdot 10^{-19}$ \\ 
    \hline
    $10^4$ & $0.318 \cdot 10^{-20}$ & $0.396 \cdot 10^{-20}$ & $0.318 \cdot 10^{-20}$ \\ 
    \hline
    $10^5$ & $0.320 \cdot 10^{-21}$ & $0.401 \cdot 10^{-21}$ & $0.316 \cdot 10^{-21}$ \\ 
    \hline
    $10^6$ & $0.336 \cdot 10^{-22}$ & $0.400 \cdot 10^{-22}$ & $0.325 \cdot 10^{-22}$ \\ 
    \hline
    $10^7$ & $0.326 \cdot 10^{-23}$ & $0.375 \cdot 10^{-23}$ & $0.330 \cdot 10^{-23}$ \\ 
    \hline
    $10^8$ & $0.332 \cdot 10^{-24}$ & $0.367 \cdot 10^{-24}$ & $0.335 \cdot 10^{-24}$ \\ 
    \hline
    $10^9$ & $0.329 \cdot 10^{-25}$ & $0.366 \cdot 10^{-25}$ & $0.329 \cdot 10^{-25}$ \\ 
    \hline
    $10^{10}$ & $0.328 \cdot 10^{-26}$ & $0.382 \cdot 10^{-26}$ & $0.339 \cdot 10^{-26}$ \\ 
    \hline
    $10^{11}$ & $0.338 \cdot 10^{-27}$ & $0.399 \cdot 10^{-27}$ & $0.326 \cdot 10^{-27}$ \\ 
    \hline
    $10^{12}$ & $0.329 \cdot 10^{-28}$ & $0.371 \cdot 10^{-28}$ & $0.327 \cdot 10^{-28}$ \\ 
    \hline
    $10^{13}$ & $0.317 \cdot 10^{-29}$ & $0.396 \cdot 10^{-29}$ & $0.333 \cdot 10^{-29}$ \\ 
    \hline
    $10^{14}$ & $0.320 \cdot 10^{-30}$ & $0.388 \cdot 10^{-30}$ & $0.325 \cdot 10^{-30}$ \\ 
    \hline
    $10^{15}$ & $0.311 \cdot 10^{-31}$ & $0.384 \cdot 10^{-31}$ & $0.336 \cdot 10^{-31}$ \\ 
    \hline
    $10^{16}$ & $0.324 \cdot 10^{-32}$ & $0.366 \cdot 10^{-32}$ & $0.321 \cdot 10^{-32}$ \\ 
    \hline
  \end{tabular}
  \end{center}
  \end{table}

  \begin{table}
   \label{normgvtable}
   \caption{Values of $\|\nabla v_3\|_0$}
  \begin{center}
  \begin{tabular}{ |c|c|c|c| }
    \hline
     $\sigma$ & Example \ref{Example_11} & Example \ref{Example_12} & Example \ref{Example_13}\\ 
    \hline
    $10^1$ & $0.941 \cdot 10^{-8}$ & $0.101 \cdot 10^{-7}$ & $0.106 \cdot 10^{-7}$ \\ 
    \hline
    $10^2$ & $0.937 \cdot 10^{-8}$ & $0.105 \cdot 10^{-7}$ & $0.102 \cdot 10^{-7}$ \\ 
    \hline
    $10^3$ & $0.920 \cdot 10^{-8}$ & $0.102 \cdot 10^{-7}$ & $0.102 \cdot 10^{-7}$ \\ 
    \hline
    $10^4$ & $0.936 \cdot 10^{-8}$ & $0.994 \cdot 10^{-8}$ & $0.104 \cdot 10^{-7}$ \\ 
    \hline
    $10^5$ & $0.953 \cdot 10^{-8}$ & $0.101 \cdot 10^{-7}$ & $0.104 \cdot 10^{-7}$ \\ 
    \hline
    $10^6$ & $0.935 \cdot 10^{-8}$ & $0.101 \cdot 10^{-7}$ & $0.105 \cdot 10^{-7}$ \\ 
    \hline
    $10^7$ & $0.946 \cdot 10^{-8}$ & $0.102 \cdot 10^{-7}$ & $0.103 \cdot 10^{-7}$ \\ 
    \hline
    $10^8$ & $0.936 \cdot 10^{-8}$ & $0.101 \cdot 10^{-7}$ & $0.103 \cdot 10^{-7}$ \\ 
    \hline
    $10^9$ & $0.942 \cdot 10^{-8}$ & $0.100 \cdot 10^{-7}$ & $0.104 \cdot 10^{-7}$ \\ 
    \hline
    $10^{10}$ & $0.934 \cdot 10^{-8}$ & $0.101 \cdot 10^{-7}$ & $0.101 \cdot 10^{-7}$ \\ 
    \hline
    $10^{11}$ & $0.931 \cdot 10^{-8}$ & $0.102 \cdot 10^{-7}$ & $0.103 \cdot 10^{-7}$ \\ 
    \hline
    $10^{12}$ & $0.926 \cdot 10^{-8}$ & $0.101 \cdot 10^{-7}$ & $0.104 \cdot 10^{-7}$ \\ 
    \hline
    $10^{13}$ & $0.939 \cdot 10^{-8}$ & $0.102 \cdot 10^{-7}$ & $0.104 \cdot 10^{-7}$ \\ 
    \hline
    $10^{14}$ & $0.940 \cdot 10^{-8}$ & $0.995 \cdot 10^{-8}$ & $0.103 \cdot 10^{-7}$ \\ 
    \hline
    $10^{15}$ & $0.945 \cdot 10^{-8}$ & $0.986 \cdot 10^{-8}$ & $0.102 \cdot 10^{-7}$ \\ 
    \hline
    $10^{16}$ & $0.920 \cdot 10^{-8}$ & $0.103 \cdot 10^{-7}$ & $0.104 \cdot 10^{-7}$ \\ 
    \hline
  \end{tabular}
  \end{center}
  \end{table}

  \begin{table}
   \label{normgwtable}
   \caption{Values of $\|\nabla w_3\|_0$}
  \begin{center}
  \begin{tabular}{ |c|c|c|c| }
    \hline
     $\sigma$ & Example \ref{Example_11} & Example \ref{Example_12} & Example \ref{Example_13}\\ 
    \hline
    $10^1$ & $0.730 \cdot 10^{-16}$ & $0.432 \cdot 10^{-16}$ & $0.844 \cdot 10^{-16}$ \\ 
    \hline
    $10^2$ & $0.728 \cdot 10^{-16}$ & $0.445 \cdot 10^{-16}$ & $0.857 \cdot 10^{-16}$ \\ 
    \hline
    $10^3$ & $0.687 \cdot 10^{-16}$ & $0.430 \cdot 10^{-16}$ & $0.813 \cdot 10^{-16}$ \\ 
    \hline
    $10^4$ & $0.712 \cdot 10^{-16}$ & $0.458 \cdot 10^{-16}$ & $0.835 \cdot 10^{-16}$ \\ 
    \hline
    $10^5$ & $0.754 \cdot 10^{-16}$ & $0.424 \cdot 10^{-16}$ & $0.848 \cdot 10^{-16}$ \\ 
    \hline
    $10^6$ & $0.727 \cdot 10^{-16}$ & $0.422 \cdot 10^{-16}$ & $0.873 \cdot 10^{-16}$ \\ 
    \hline
    $10^7$ & $0.716 \cdot 10^{-16}$ & $0.444 \cdot 10^{-16}$ & $0.773 \cdot 10^{-16}$ \\ 
    \hline
    $10^8$ & $0.723 \cdot 10^{-16}$ & $0.444 \cdot 10^{-16}$ & $0.859 \cdot 10^{-16}$ \\ 
    \hline
    $10^9$ & $0.727 \cdot 10^{-16}$ & $0.420 \cdot 10^{-16}$ & $0.833 \cdot 10^{-16}$ \\ 
    \hline
    $10^{10}$ & $0.721 \cdot 10^{-16}$ & $0.431 \cdot 10^{-16}$ & $0.752 \cdot 10^{-16}$ \\ 
    \hline
    $10^{11}$ & $0.698 \cdot 10^{-16}$ & $0.478 \cdot 10^{-16}$ & $0.808 \cdot 10^{-16}$ \\ 
    \hline
    $10^{12}$ & $0.672 \cdot 10^{-16}$ & $0.431 \cdot 10^{-16}$ & $0.843 \cdot 10^{-16}$ \\ 
    \hline
    $10^{13}$ & $0.692 \cdot 10^{-16}$ & $0.449 \cdot 10^{-16}$ & $0.785 \cdot 10^{-16}$ \\ 
    \hline
    $10^{14}$ & $0.733 \cdot 10^{-16}$ & $0.414 \cdot 10^{-16}$ & $0.789 \cdot 10^{-16}$ \\ 
    \hline
    $10^{15}$ & $0.739 \cdot 10^{-16}$ & $0.430 \cdot 10^{-16}$ & $0.800 \cdot 10^{-16}$ \\ 
    \hline
    $10^{16}$ & $0.687 \cdot 10^{-16}$ & $0.442 \cdot 10^{-16}$ & $0.779 \cdot 10^{-16}$ \\ 
    \hline
  \end{tabular}
  \end{center}
  \end{table}

  \begin{table}
   \label{normhvtable}
   \caption{Values of $\|\nabla^2 v_3\|_0$}
  \begin{center} 
  \begin{tabular}{ |c|c|c|c| }
    \hline
     $\sigma$ & Example \ref{Example_11} & Example \ref{Example_12} & Example \ref{Example_13}\\ 
    \hline
    $10^1$ & $3.37 \cdot 10^{13}$ & $3.05 \cdot 10^{13}$ & $3.26 \cdot 10^{13}$ \\ 
    \hline
    $10^2$ & $3.30 \cdot 10^{15}$ & $3.00 \cdot 10^{15}$ & $3.22 \cdot 10^{15}$ \\ 
    \hline
    $10^3$ & $3.27 \cdot 10^{17}$ & $2.98 \cdot 10^{17}$ & $3.28 \cdot 10^{17}$ \\ 
    \hline
    $10^4$ & $3.26 \cdot 10^{19}$ & $2.98 \cdot 10^{19}$ & $3.27 \cdot 10^{19}$ \\ 
    \hline
    $10^5$ & $3.27 \cdot 10^{21}$ & $2.99 \cdot 10^{21}$ & $3.29 \cdot 10^{21}$ \\ 
    \hline
    $10^6$ & $3.25 \cdot 10^{23}$ & $2.99 \cdot 10^{23}$ & $3.27 \cdot 10^{23}$ \\ 
    \hline
    $10^7$ & $3.28 \cdot 10^{25}$ & $2.99 \cdot 10^{25}$ & $3.23 \cdot 10^{25}$ \\ 
    \hline
    $10^8$ & $3.21 \cdot 10^{27}$ & $2.99 \cdot 10^{27}$ & $3.28 \cdot 10^{27}$ \\ 
    \hline
    $10^9$ & $3.26 \cdot 10^{29}$ & $2.98 \cdot 10^{29}$ & $3.28 \cdot 10^{29}$ \\ 
    \hline
    $10^{10}$ & $3.30 \cdot 10^{31}$ & $2.99 \cdot 10^{31}$ & $3.30 \cdot 10^{31}$ \\ 
    \hline
    $10^{11}$ & $3.25 \cdot 10^{33}$ & $2.99 \cdot 10^{33}$ & $3.23 \cdot 10^{33}$ \\ 
    \hline
    $10^{12}$ & $3.28 \cdot 10^{35}$ & $2.99 \cdot 10^{35}$ & $3.24 \cdot 10^{35}$ \\ 
    \hline
    $10^{13}$ & $3.23 \cdot 10^{37}$ & $2.99 \cdot 10^{37}$ & $3.25 \cdot 10^{37}$ \\ 
    \hline
    $10^{14}$ & $3.26 \cdot 10^{39}$ & $2.98 \cdot 10^{39}$ & $3.25 \cdot 10^{39}$ \\ 
    \hline
    $10^{15}$ & $3.25 \cdot 10^{41}$ & $2.99 \cdot 10^{41}$ & $3.27 \cdot 10^{41}$ \\ 
    \hline
    $10^{16}$ & $3.29 \cdot 10^{43}$ & $2.98 \cdot 10^{43}$ & $3.20 \cdot 10^{43}$ \\ 
    \hline
  \end{tabular}
  \end{center}
  \end{table}

  \begin{table}
   \label{normhwtable}
   \caption{Values of $\|\nabla^2 w_3\|_0$}
  \begin{center}
  \begin{tabular}{ |c|c|c|c| }
    \hline
     $\sigma$ & Example \ref{Example_11} & Example \ref{Example_12} & Example \ref{Example_13}\\ 
    \hline
    $10^1$ & $3.11 \cdot 10^{5}$ & $2.80 \cdot 10^{5}$ & $2.12 \cdot 10^{5}$ \\ 
    \hline    
    $10^2$ & $3.12 \cdot 10^{7}$ & $2.82 \cdot 10^{7}$ & $2.05 \cdot 10^{7}$ \\ 
    \hline
    $10^3$ & $3.20 \cdot 10^{9}$ & $2.87 \cdot 10^{9}$ & $2.23 \cdot 10^{9}$ \\ 
    \hline
    $10^4$ & $3.17 \cdot 10^{11}$ & $2.60 \cdot 10^{11}$ & $2.18 \cdot 10^{11}$ \\ 
    \hline
    $10^5$ & $3.12 \cdot 10^{13}$ & $2.85 \cdot 10^{13}$ & $2.15 \cdot 10^{13}$ \\ 
    \hline
    $10^6$ & $3.18 \cdot 10^{15}$ & $2.80 \cdot 10^{15}$ & $2.08 \cdot 10^{15}$ \\ 
    \hline
    $10^7$ & $3.12 \cdot 10^{17}$ & $2.79 \cdot 10^{17}$ & $2.18 \cdot 10^{17}$ \\ 
    \hline
    $10^8$ & $3.17 \cdot 10^{19}$ & $2.73 \cdot 10^{19}$ & $2.18 \cdot 10^{19}$ \\ 
    \hline
    $10^9$ & $3.27 \cdot 10^{21}$ & $2.74 \cdot 10^{21}$ & $2.18 \cdot 10^{21}$ \\ 
    \hline
    $10^{10}$ & $3.31 \cdot 10^{23}$ & $2.75 \cdot 10^{23}$ & $2.10 \cdot 10^{23}$ \\ 
    \hline
    $10^{11}$ & $3.21 \cdot 10^{25}$ & $2.65 \cdot 10^{25}$ & $2.13 \cdot 10^{25}$ \\ 
    \hline
    $10^{12}$ & $3.25 \cdot 10^{27}$ & $2.69 \cdot 10^{27}$ & $2.19 \cdot 10^{27}$ \\ 
    \hline
    $10^{13}$ & $3.10 \cdot 10^{29}$ & $2.79 \cdot 10^{29}$ & $2.08 \cdot 10^{29}$ \\ 
    \hline
    $10^{14}$ & $3.08 \cdot 10^{31}$ & $2.71 \cdot 10^{31}$ & $2.21 \cdot 10^{31}$ \\ 
    \hline
    $10^{15}$ & $3.29 \cdot 10^{33}$ & $2.79 \cdot 10^{33}$ & $2.20 \cdot 10^{33}$ \\ 
    \hline
    $10^{16}$ & $3.39 \cdot 10^{35}$ & $2.72 \cdot 10^{35}$ & $2.22 \cdot 10^{35}$ \\ 
    \hline
   \end{tabular}
  \end{center}
  \end{table}

\end{document}